\documentclass{amsproc}
\usepackage{euscript, graphicx, epstopdf}
\usepackage{cases}
\usepackage{mathrsfs}
\usepackage{bbm}
\usepackage{amssymb}
\usepackage{txfonts}
\usepackage{amscd}
\usepackage{amsfonts,latexsym,amsmath,amsxtra,mathdots,latexsym,mathabx}
\usepackage[all,cmtip]{xy}
\usepackage{color}
\usepackage{colordvi}
\usepackage{multicol}
\usepackage{hyperref}
\usepackage{tikz}
\usepackage{float}
\usepackage{setspace, floatflt}

\allowdisplaybreaks

\DeclareFontFamily{U}{matha}{\hyphenchar\font45}
\DeclareFontShape{U}{matha}{m}{n}{
	<5> <6> <7> <8> <9> <10> gen * matha
	<10.95> matha10 <12> <14.4> <17.28> <20.74> <24.88> matha12
}{}
\DeclareSymbolFont{matha}{U}{matha}{m}{n}

\DeclareMathSymbol{\Lt}{3}{matha}{"CE}
\DeclareMathSymbol{\Gt}{3}{matha}{"CF}

\DeclareFontFamily{U}{mathb}{\hyphenchar\font45}
\DeclareFontShape{U}{mathb}{m}{n}{
	<-6> mathb5 <6-7> mathb6 <7-8> mathb7
	<8-9> mathb8 <9-10> mathb9
	<10-12> mathb10 <12-> mathb12
}{}
\DeclareSymbolFont{mathb}{U}{mathb}{m}{n}
\DeclareMathSymbol{\llcurly}{\mathrel}{mathb}{"CE}
\DeclareMathSymbol{\ggcurly}{\mathrel}{mathb}{"CF}

\DeclareSymbolFont{mathc}{OML}{txmi}{m}{it}% txfonts
\DeclareMathSymbol{\varvv}{\mathord}{mathc}{118} 
\DeclareMathSymbol{\varnu}{\mathord}{mathc}{"17}

\newcommand{\BC}{{\mathbb {C}}}

\newcommand{\BQ}{{\mathbb {Q}}} \newcommand{\BR}{{\mathbb {R}}}

\newcommand{\GL}{{\mathrm {GL}}} \newcommand{\PGL}{{\mathrm {PGL}}}
\newcommand{\SL}{{\mathrm {SL}}} 

\newcommand{\SU}{{\mathrm{SU}}} \newcommand{\Tr}{{\mathrm{Tr}}}

\newcommand{\ov}{\overline}

\newcommand{\ra}{\rightarrow}

\def\fra{\mathfrak{a}}

\def\-{^{-1}}

\def\lp {\left (}
\def\rp {\right )}

\def\SSS{\text{\mbox{\larger[1]$\text{\usefont{U}{BOONDOX-calo}{m}{n}S}$}}\hskip 1pt}
\def\SS{\raisebox{- 2 \depth}{$\SSS$}}
\def\SSH{\text{\mbox{\larger[1]$\text{\usefont{U}{BOONDOX-calo}{m}{n}H}$}}\hskip 1pt}
\def\SH{\raisebox{- 2 \depth}{$\SSH$}}
\def\SSB{\text{\mbox{\larger[1]$\text{\usefont{U}{BOONDOX-calo}{m}{n}B}$}}\hskip 1pt}
\def\SB{\raisebox{- 3 \depth}{$\SSB$}}
\def\boldJ {\boldsymbol J}

\renewcommand{\Im}{{\mathrm{Im}\,}}
\renewcommand{\Re}{{\mathrm{Re}\,}}

\def\bfJ {\boldsymbol J}

\def\bfD {\boldsymbol D}
\def\bfE {\boldsymbol E}
\def\bfF {\boldsymbol F}
\def\bfG {\boldsymbol G}
\def\bfI {\boldsymbol I}
\def\bfM {\boldsymbol M}
\def\bfP {\boldsymbol P}
\def\bfR {\boldsymbol R}

\def\nwedge {\hskip - 2 pt \wedge \hskip - 2 pt }
\def\shskip{\hskip 0.5 pt}
\def\tw{\textit{w}}

\def\tv{\textit{v}}

\newcommand{\bfrac}[2]{\text{\Large$\frac {#1} {#2}$}}

\makeatletter
\g@addto@macro\normalsize{\setlength\abovedisplayskip{3pt}}
\makeatother

\makeatletter
\g@addto@macro\normalsize{\setlength\belowdisplayskip{3pt}}
\makeatother

\newcommand{\delete}[1]{}

\theoremstyle{plain}

\newtheorem{thm}{Theorem}[section] \newtheorem{cor}[thm]{Corollary}
\newtheorem{lem}[thm]{Lemma}  \newtheorem{prop}[thm]{Proposition}

\newtheorem {rem}[thm]{Remark}

\numberwithin{equation}{section}
\newtheorem*{acknowledgement}{Acknowledgements}

\begin{document}

	\title[Bessel functions and Beyond Endoscopy]{On the Fourier transform of regularized Bessel functions on complex numbers and Beyond Endoscopy over number fields}

	\author{Zhi Qi}
	\address{School of Mathematical Sciences\\ Zhejiang University\\Hangzhou, 310027\\China}
	\email{zhi.qi@zju.edu.cn}


	\begin{abstract}
		In this article, we prove certain Weber-Schafheitlin type integral formulae for Bessel functions over complex numbers. A special case is a formula for the Fourier transform of  regularized Bessel functions on complex numbers. This is applied to extend the work of A. Venkatesh  on Beyond Endoscopy for $\mathrm{Sym}^2$  on $\mathrm{GL}_2$ from totally real to arbitrary number fields. 
	\end{abstract}
	
	\maketitle

	\section{Introduction}
	
	A special case of the discontinuous integrals of Weber and Schafheitlin on the Fourier transform of Bessel functions on $\BR_+ = (0, \infty)$ is as follows (see \cite[13.42 (2), (3)]{Watson}),
	\begin{equation}\label{0eq: Weber-Sch}
\int_0^{\infty} \frac {J_{\varnu } (4 \pi x) e (\pm x y)} {x} d x = \left\{ \begin{split}
&\frac { 2^{\varnu }   }  {\varnu  \big(\sqrt {4 -y^2  } \mp i y \big)^{\varnu  } }, \hskip 10 pt \text{ if } 0 \leqslant y \leqslant 2, \\
& \frac { 2^{\varnu } e^{\pm \frac 1 2 \pi i \varnu }  }  {\varnu  \big( \sqrt {y^2 - 4} + y \big)^{\varnu  } }, \hskip 12.8 pt  \text{ if } y > 2,
\end{split} \right.
	\end{equation}
valid for $\Re \varnu  > 0$, where $J_{\varnu } (z)$ is the Bessel function of the first kind of order $\varnu $ and as usual $e (z) = e^{2 \pi i z}$; it is assumed here that $\arg \big(\hskip -1 pt \sqrt {4 - y^2} \mp i y\big) \in \mp \left[0, \tfrac 1 2 \pi \right]$.\footnote{The first line on the right hand side of  \eqref{0eq: Weber-Sch} should read $(1/\varnu ) \exp \left(\pm i \varnu  \arcsin (y/2) \right)$ by bookkeeping \cite[13.42 (2), (3)]{Watson}, but we feel that the formulation here in  \eqref{0eq: Weber-Sch} is more suggestive.} Subsequently, we shall call a formula of this kind  {\it special Weber-Schafheitlin integral formula}.

In Venkatesh's work on Beyond Endoscopy  \cite{Venkatesh-BeyondEndoscopy}, the second formula in  \eqref{0eq: Weber-Sch} arises in his local computation over $\BR$, particularly, in his analysis for the $B$-transform. The results of Venkatesh were proven for {\it totally real} number fields, but he pointed out that the extension to complex places would only require verifying  a similar formula  for Bessel functions over $\BC$.  This amounts to a ``local fundamental lemma" over $\BC$. Unfortunately, it seems to resist a proof in every direct way---Venkatesh was not able to prove it at that time.

The purpose of this paper is to establish an integral formula for Bessel functions over complex numbers which is analogous to the special  Weber-Schafheitlin formula as in \eqref{0eq: Weber-Sch} (after regularization). Our approach is a rather indirect method that combines asymptotic analysis and differential equations. As an application, the validity of Venkatesh's work on Beyond Endoscopy is extended from totally real to arbitrary number fields. 

\subsection{Analysis over real numbers} \label{sec: analysis over R} \

\vskip 5 pt

\subsubsection*{Bessel kernels and the special Weber-Schafheitlin formula} 

First of all, some remarks on the special Weber-Schafheitlin formula \eqref{0eq: Weber-Sch} are in order, as the motivation of our investigation of its complex analogue.  

The Bessel function of concern has pure imaginary order $\varnu  = \pm 2it$ ($t$ real). Indeed, the Bessel kernel  in the $B$-transform is $B_{2it} (4 \pi x)$  (see \cite[\S 6.7.2]{Venkatesh-BeyondEndoscopy}), with
\begin{align}
B_{\varnu } (  x) =  \frac {1} {\sin (\pi \varnu /2)} \lp J_{-\varnu } (  x) - J_{\varnu } (  x) \rp, \quad x \in \BR_+,
\end{align}
and the $B$-transform is defined by
\begin{align}
\phi (x) = \frac 1 2 \int_{-\infty}^{\infty} B_{2it} (4 \pi x) h(t) \sinh (\pi t) t d t.
\end{align}

The formula  \eqref{0eq: Weber-Sch} however is only valid for $\Re \varnu  > 0$. The reason is that $J_{\varnu } (4 \pi x) \sim (2 \pi x )^{\shskip\varnu }/ \Gamma (\varnu +1)$ as $x \ra 0$ (see \cite[3.1 (8)]{Watson}), say for $\Re \varnu  \geqslant 0$, and hence %\footnote{Recall that $J_{-n} (z) = (-)^n J_n (z)$ for $n = 1, 2, ...$.} and $ J_{\varnu } (4 \pi x) = O (1/\sqrt x)$, 
the integral in \eqref{0eq: Weber-Sch} is convergent near zero only when $\Re \varnu  > 0$. It should be noted that the integral is always absolutely convergent in the vicinity of infinity since $ J_{\varnu } (4 \pi x) = O (1/\sqrt x)$ as $x \ra \infty$ (see \cite[7.21 (1)]{Watson}).
Nevertheless, this convergence issue  may be easily addressed by shifting  $ t $ to $ t - i \sigma$ or $t + i \sigma $  ($\sigma > 0$) for $J_{  2 i t} (4 \pi x)$ or  $J_{- 2 i t} (4 \pi x)$ respectively. 
%We now denote by $M_{\varnu } (\pm y)$ the function on the right hand side of \eqref{0eq: Weber-Sch}. Then $M_{\varnu } (y)$ is continuous in $y$ and analytic for all $\varnu $ except when $\varnu  = 0$.

The Bessel kernel  in the $K$-transform  is $  (4/ {\pi}) \cos (\pi t) K_{2it} (4 \pi x)$ (see \cite[\S 6.7.3]{Venkatesh-BeyondEndoscopy}), with
\begin{align}\label{0eq: K-Bessel}
 K_{\varnu } (x) =   \frac {\pi} { 2 \sin (\pi \varnu  )} \lp I_{-\varnu } (  x) - I_{\varnu } (x) \rp,  \quad x \in \BR_+,
\end{align}
where $I_{\varnu } (z)$ and $K_{\varnu } (z)$ are the modified Bessel functions\footnote{Note that the $\sin (\pi \varnu )$ is mistakenly written as $\sin (\pi \varnu /2)$  in \cite{Venkatesh-BeyondEndoscopy}.}. The $K$-transform is defined by
\begin{align}
\phi (x) = \frac 1 {\pi} \int_{-\infty}^{\infty} K_{2it} (4 \pi x) h(t) \sinh (2 \pi t) t d t.
\end{align}

The integral in \eqref{0eq: Weber-Sch} would diverge if the $J_{\varnu } (4 \pi x)$ were replaced by  $I_{\varnu } (4 \pi x)$ since it is of exponential growth ($I_{\varnu } (x) \sim \exp (x) / \sqrt {2\pi x}$, \cite[7.23 (2), (3)]{Watson}) as $x \ra \infty$. Although $I_{\varnu } (x)$ and $I_{- \varnu } (x)$ conspire in \eqref{0eq: K-Bessel} so that $ K_{\varnu } (x)$ decays exponentially ($K_{\varnu } (x) \sim \exp (- x) / \sqrt { 2 x / \pi}$,  \cite[7.23 (1)]{Watson})  as $x \ra \infty$, the integral in \eqref{0eq: Weber-Sch} would still diverge if the $J_{\varnu } (4 \pi x)$ were replaced by  $K_{\varnu } (4 \pi x)$. This is because $ I_{\varnu } (4 \pi x) \sim (2 \pi x )^{\shskip\varnu }/ \Gamma (\varnu +1) $ as $x \ra 0$ (see \cite[3.7 (2)]{Watson}) and it would require both $\Re \varnu  > 0$ and $\Re (-\varnu ) > 0$ for the integral to be convergent near zero, which is clearly impossible. At any rate, the special Weber-Schafheitlin formula does {\it not} exist for the Bessel kernel of the $K$-transform. Moreover, in the case $\varnu  = 2 i t$, the order-shifting trick as above applied to $I_{2it} (4\pi x)$ and  $I_{-2it} (4\pi x)$ would not work: after the order-shifting their exponential growth could not be canceled completely and the resulting difference $I_{-2it+2 \sigma} (4\pi x) - I_{2it+2 \sigma} (4\pi x)$ ($\sigma > 0$) still grows exponentially (compare their asymptotic expansions as in \cite[7.23 (2), (3)]{Watson}).  As such, the analysis for  the $K$-transform in \cite[\S 6.7.3]{Venkatesh-BeyondEndoscopy} is quite different from that for the $B$-transform in \cite[\S 6.7.2]{Venkatesh-BeyondEndoscopy}. Instead, Venkatesh uses the Lebedev-Kontorovitch inversion and the Mehler-Sonine integral representation for $K_{2it} (x)$.

Similar obstacles described above for $K_{2it} (4 \pi x) $ would arise when analyzing the Fourier transform of Bessel functions on $\BC$. Thus before we initiate our study over $\BC$ it is necessary to overcome these obstacles over $\BR$ with new ideas. %This is achieved by regularization. 

\vskip 5 pt

\subsubsection*{Regularizing the special Weber-Schafheitlin formula}
%	It is often in the literature of analytic number theory that the analysis for $B$-transform and that for $K$-transform are entirely different. 

Next, we introduce the regularized Bessel kernels and their special Weber-Schafheitlin formulae and briefly explain how they are applied to unify the analysis for the $B$-transform and that for the $K$-transform. The merit of  unified analyses for  the $B$- and the $K$-transform and also for their Bessel kernels is that their complex extensions are usually admissible.

We first define
\begin{align}
P_{\varnu } (x) = \frac {(x/2)^{\varnu } } {\Gamma (\varnu  + 1)}.
\end{align}
Note that $P_{\varnu } (x)  $ is simply the leading term in the series expansion of $J_{\varnu } (x)$ or  $I_{\varnu } (x)$ at $x = 0$ (see \cite[3.1 (8), 3.7 (2)]{Watson}).
By \cite[3.764]{G-R}, we have 
\begin{align}\label{0eq: Fourier of x}
\int_0^{\infty} x^{\shskip\varnu  - 1} e (\pm x y) d x =  \Gamma (\varnu ) e^{\pm \frac 1 2 \pi i \varnu } (2\pi   y)^{-\varnu }, \hskip 10 pt  0 < \Re \varnu  < 1.
\end{align}
%Note that it has a complex analogue in Lemma \ref{lem: integral of P}.
%This integral is convergent after integration by parts near infinity, and the convergence is compactly uniform in $\varnu $.
%If the integral is localized by a weight function supported in a neighborhood of $\infty$, say from a partition of unity that separates $0$ and $\infty$, then the weighted integral is convergent for all $\Re \varnu  < 1$ (partial integration  is needed for $0 \leqslant \Re \varnu  < 1$) and the convergence is compactly uniform in $\varnu $. 

In the following, let $\Re \varnu  = 0$ and, for simplicity,  $\varnu  \neq 0$. The results may be extended to $ \varnu  = 0 $ and even to all $\varnu $ with $|\Re \varnu | < 1 $ by the principle of analytic continuation. 

We now define
\begin{align}\label{0eq: defn of R(x)}
R_{\varnu } (x) = \frac {1} {\sin (\pi \varnu /2) } \lp P_{- \varnu } (x) - P_{\varnu } (x)  \rp ,
\end{align}
and the regularized Bessel kernels
\begin{align}
D_{\varnu } (x) = B_{\varnu } (  x) - R_{\varnu } (x), \hskip 10 pt M_{\varnu } (x) =  (4/ {\pi}) \cos (\pi \varnu /2) K_{\varnu } (  x) - R_{\varnu } (x).
\end{align}
Consider the following two integrals
\begin{align}\label{0eq: regularized integrals}
\int_0^{\infty} \frac {D_{\varnu } (4 \pi x) e (\pm x y)} {x} d x, \hskip 10 pt \int_0^{\infty} \frac {M_{\varnu } (4 \pi x) e (\pm x y)} {x} d x. 
\end{align}
We have $J_{\varnu } (x) - P_{\varnu } (x)$, $I_{\varnu } (x) - P_{\varnu } (x) = O  ( x^2  )$ as $x \ra 0$. Thus these integrals are absolutely convergent at zero and become so after integration by parts in the vicinity of infinity (the region of convergence may actually be extended to $|\Re \varnu  | < 1$ and the convergence is uniform for $\Re \varnu $ in a compact set). %, .

For the first integral in \eqref{0eq: regularized integrals}, we apply \eqref{0eq: Weber-Sch} and \eqref{0eq: Fourier of x} to the Fourier transform of $\lp J_{-\varnu  +\sigma} (4\pi x) - P _{-\varnu  +\sigma} (4\pi x) \rp / x$ and $\lp J_{\varnu  +\sigma} (4\pi x) - P _{\varnu  +\sigma} (4\pi x) \rp/ x$ and then let $\sigma \ra 0$; the uniformity in $\Re \sigma$ may be easily verified.  In this way, we obtain for $y > 2$ that
\begin{equation}\label{0eq: Weber-Sch regularized}
\begin{split}
\int_0^{\infty} \frac {D_{\varnu } (4 \pi x) e (\pm x y)} {x} d x = - \frac {1} {\varnu  \sin (\pi \varnu /2)} \Bigg\{ & e^{\mp \frac 1 2 \pi i \varnu } \lp \frac {\big(  y + \sqrt {y^2 - 4} \big)^{\varnu  } } { 2^{\varnu }   }   -   {y^{\shskip\varnu }} \rp \\
 + \, & e^{\pm \frac 1 2 \pi i \varnu } \lp \frac { 2^{\varnu }   }  {\big( y + \sqrt {y^2 - 4} \big)^{\varnu  } } - \frac {1} {y^{\shskip\varnu }} \rp \Bigg\}.
\end{split}
\end{equation}
The formula in the case $0 \leqslant y \leqslant 2$ is similar. 

For the second integral in \eqref{0eq: regularized integrals}, as indicated before, there is no   special Weber-Schafheitlin formula  for either $I_{\varnu } (4 \pi x)$ or $K_{\varnu } (4 \pi x) $. % analogous to  \eqref{0eq: Weber-Sch}. 
We propose an alternative approach by modifying the integrals by the factor $x^{\shskip\rho}$ with $ 0 < \Re \rho < 1$. By \cite[6.699 3, 4]{G-R} and the transformation formula for hypergeometric functions with respect to $z \ra 1/z$ (see \cite[\S 2.4.1]{MO-Formulas}), along with the reflection and the duplication formula for the gamma function, we infer that for all $\Re \rho > 0$
\begin{equation}\label{0eq: Weber-Sch for K}
\begin{split}
 \int_0^\infty  K_{\varnu } (4 \pi x)  e  (\pm x y)  & x^{\shskip \rho - 1 }  d x =    \\
 \frac {\pi} {2 (2 \pi)^{  \rho} \sin (\pi \varnu )}  \bigg\{ & \frac {       \Gamma (\rho-\varnu ) e^{\pm \frac 1 2 \pi i (\rho-\varnu )} } {  \Gamma (1-\varnu ) y^{\shskip \rho - \varnu }    }  {_2F_1} \hskip -2 pt \lp \frac {\rho-\varnu } 2 , \frac {1+\rho - \varnu } 2 ; 1-\varnu   ; - \frac 4 {y^2}  \rp  \\
   - & \frac {       \Gamma (\rho+\varnu ) e^{\pm \frac 1 2 \pi i (\rho+\varnu )}  } {  \Gamma (1+\varnu ) y^{\shskip \rho + \varnu }    }  {_2F_1} \hskip -2 pt \lp \frac {\rho + \varnu } 2, \frac {1+\rho+\varnu } 2  ; 1+ \varnu  ; - \frac 4 {y^2}  \rp \bigg\}.
\end{split}
\end{equation}
A formula of this kind will be called {\it general Weber-Schafheitlin integral formula}. We now apply \eqref{0eq: Fourier of x} and \eqref{0eq: Weber-Sch for K} to the Fourier transform of $(4/\pi) \cos (\pi \varnu /2) K_{\varnu } (4 \pi x) x^{\shskip \rho-1} -    \big(  P_{\rho - \varnu } (4 \pi x) - P_{\rho+\varnu } (4 \pi x)  \big) /  {\sin (\pi \varnu /2) }$ and then let $\rho \ra 0$. The resulting hypergeometric functions may be evaluated by the formula (see \cite[\S 2.1]{MO-Formulas})
\begin{align}\label{0eq: hypergeometric, evalation}
\big( 1 + \sqrt {1-z^2} \big)^{- 2 a} = 2^{-2 a} {_2F_1} \hskip -2 pt \lp a , a +   1 / 2 ; 2 a+1  ; z^2  \rp.
\end{align}
Finally we obtain
\begin{equation}\label{0eq: Weber-Sch regularized, 2}
\begin{split}
\int_0^{\infty} \frac {M_{\varnu } (4 \pi x) e (\pm x y)} {x} d x = - \frac {1} {\varnu  \sin (\pi \varnu /2)} \Bigg\{ & e^{\mp \frac 1 2 \pi i \varnu } \lp \frac {\big(y + \sqrt {y^2 + 4}  \big)^{\varnu  } } { 2^{\varnu }   }   -   {y^{\shskip\varnu }} \rp \\
+ \, & e^{\pm \frac 1 2 \pi i \varnu } \lp \frac { 2^{\varnu }   }  {\big(y + \sqrt {y^2 + 4} \big)^{\varnu  } } - \frac {1} {y^{\shskip\varnu }} \rp \Bigg\},
\end{split}
\end{equation}
which is very similar to \eqref{0eq: Weber-Sch regularized}. 
This somewhat indirect approach to \eqref{0eq: Weber-Sch regularized, 2} would also lead us to \eqref{0eq: Weber-Sch regularized}; we only record here the general Weber-Schafheitlin integral formula for $J_{\varnu } (4 \pi x)$ as below in the case $y > 2$ (see \cite[6.699 1, 2]{G-R}),
\begin{equation}\label{0eq: Weber-Sch for J}
\begin{split}
\int_0^\infty  J_{\varnu } (4 \pi x)  e  (\pm x y)  & x^{\shskip \rho - 1 }  d x = \frac {\Gamma (\rho+\varnu ) e^{\pm \frac 1 2 \pi i (\rho+\varnu )}  } {(2\pi)^{ \rho} \Gamma (1+\varnu ) y^{\shskip \rho+\varnu  } } {_2F_1} \hskip -2 pt \lp \frac {\rho + \varnu } 2, \frac {1+\rho+\varnu } 2  ; 1+ \varnu  ;  \frac 4 {y^2}  \rp \hskip -1 pt,
\end{split}
\end{equation}
valid for $0 < \Re \rho < \bfrac 3 2$.

An observation is that the paired functions in \eqref{0eq: Weber-Sch regularized} and \eqref{0eq: Weber-Sch regularized, 2}, say $ \big( y + \sqrt {y^2 - 4} \big)^{\varnu  }  / 2^{\varnu }   $ and $   {y^{\shskip\varnu }}  $, are asymptotically equivalent as $y \ra \infty$. In other words, the Fourier transform of $B_{\varnu } ( 4 \pi x) / x$ or $ 4 \cos (\pi \varnu /2) K_{\varnu } ( 4 \pi x) / \pi x  $ and that of $ R_{\varnu } (4 \pi x) / x $, in the formal sense, have the same asymptotic at infinity. Note that this is also true if they are modified by the factor $x^{\shskip \rho}$. A similar observation, Lemma \ref{lem: F sim G}, is crucial to our analysis over complex numbers. See \S \ref{sec: asymptotic} for more details.

For the analysis for the $B$-transform  in \cite[\S 6.7.2]{Venkatesh-BeyondEndoscopy}, we modify the arguments therein as follows. First, we start with writing
\begin{equation}\label{1eq: Fourier of Bessel, real, 1}
\begin{split}
\int_0^{\infty} \cos (2\pi k x) \phi (x) \frac {d x} x =  \frac 1 2 & \int_0^{\infty} \cos (2\pi k x) \int_{-\infty}^{\infty} D_{2it} (4 \pi x) h(t) \sinh (\pi t) t d t \frac {d x} x \\
 +   i & \int_0^{\infty} \cos (2\pi k x) \int_{-\infty}^{\infty} P_{2it} (4 \pi x) h(t) t d t \frac {d x} x,
\end{split}
\end{equation}
instead of 
\begin{align}\label{1eq: Fourier of Bessel, real, 2}
	\int_0^{\infty} \cos (2\pi k x) \phi (x) \frac {d x} x =    i \int_0^{\infty} \cos (2\pi k x) \int_{-\infty}^{\infty} J_{2it} (4 \pi x) h(t) t d t \frac {d x} x. 
\end{align}
Second, we apply \eqref{0eq: Weber-Sch regularized} directly to the first integral on the right of \eqref{1eq: Fourier of Bessel, real, 1}, 
 and shift the order of $ P_{2 i t } (4 \pi x)$ in \eqref{1eq: Fourier of Bessel, real, 1} instead of that of  $ J_{2it } (4 \pi x)$ in \eqref{1eq: Fourier of Bessel, real, 2}. Note that  the $B$-transform turns into the inverse Mellin transform if % its Bessel kernel $B_{2it } (4 \pi x)$ (or 
 $J_{2it} (4 \pi x)$ is substituted by % $ R_{2it }  (4 \pi x)$ (or 
 $P_{2it }  (4 \pi x)$. It is clear that the analysis for the $K$-transform may also be done in this way. Since these will be elaborated in the complex setting in \S  \ref{sec: Proof of Venkatesh}, we do not discuss here any further details. %The reader should have no difficulty in making suitable changes for the real case. 

The ideas outlined above will be executed  in \S \S \ref{sec: regularized W-S} and \ref{sec: Proof of Venkatesh} for our analysis on complex numbers. Yet we still need to first establish the complex analogue of the general Weber-Schafheitlin integral formula in \eqref{0eq: Weber-Sch for K} or \eqref{0eq: Weber-Sch for J}. This is actually the main technical result of this article; its proof will be given mainly in \S \ref{sec: Proof of Theorem 1}.  %This is the technical part.

\subsection{Main theorems}

We now introduce the definition of Bessel functions over complex numbers (see \cite[\S 15.3]{Qi-Bessel}, \cite[(6.21), (7.21)]{B-Mo}). Let $\mu $ be a complex number  and $d  $ be an integer. 
We define 
\begin{equation}\label{0def: J mu m (z)}
J_{\mu ,\shskip  d} (z) = J_{\mu + d } \lp  z \rp J_{\mu -  d  } \lp  {\widebar z} \rp.
\end{equation}
The function $J_{\mu,\shskip  d} (z)$ is well defined and {even} on $\BC \smallsetminus \{0\}$ in the sense that the   expression on the right of \eqref{0def: J mu m (z)} is independent on the choice of the argument of $z$ modulo $ \pi$. Next, we define
\begin{equation}\label{0eq: defn of Bessel}
\bfJ_{ \mu,\shskip  d} (z) =  \frac {1} {\sin (\pi \mu)} \lp J_{-\mu,\shskip  -d} (4 \pi   z) -  J_{\mu,\shskip  d} (4 \pi   z) \rp .
\end{equation} 
It is understood that in the non-generic case when  $ \mu$ is an integer the right hand side should be replaced by its limit. It is clear that $\bfJ_{ \mu,\shskip  d} (z)$ is also an even function on  $\BC \smallsetminus \{0\}$.  %$\boldJ_{\mu,\shskip m}(z)$ is an even or odd function according as $m$ is even or odd. Moreover, s
%Since $ \bfJ_{ - \mu,\shskip  - d} \lp z \rp = \bfJ_{ \mu,\shskip   d} \lp z \rp $, we may assume  with no loss of generality that $d \geqslant$.

According to \cite[\S 18.2]{Qi-Bessel}, these Bessel functions are attached to representations of $\PGL_2 (\BC)$. 
We shall not restrict ourselves to the Bessel functions of trivial $\SU_2$-type ($d = 0$) arising in the Kuznetsov-Bruggeman-Miatello  formula (see  \cite[\S 2.6, Appendix]{Venkatesh-BeyondEndoscopy} or \cite{BM-Kuz-Spherical})\footnote{For succinctness, we shall suppress $d$ from our notation if $d = 0$, so in particular $ \bfJ_{ \mu} (z) = \bfJ_{ \mu, \shskip  0} (z) $.}. For Bessel functions for  $\GL_2 (\BC)$ with non-trivial central characters, our results and method would still be valid but the formulae would be more involved. For these we refer the reader to Appendix \ref{appendix}.

\vskip 5 pt

\subsubsection*{General Weber-Schafheitlin formula}

First, we have the general Weber-Schafheitlin integral formula for $\bfJ_{ \mu,\shskip  d} (z)$ as follows. 

\begin{thm}\label{thm: general W-S formula} Suppose that
	$ |\Re \mu| < \Re \rho < \bfrac 1 2 $. Define
	\begin{equation}\label{0eq: defn of C}
	C_{ \mu,\shskip  d}^{\shskip \rho} = - \frac {  \sin (\pi ( \rho  +  \mu))  } {(2 \pi  )^{ 2 \rho} \sin (\pi \mu)}  \frac {      \Gamma (\rho +   \mu  + d) \Gamma ( \rho  +   \mu - d)   } {   \Gamma  (1 +  \mu + d ) \Gamma  (1 +  \mu - d  )},
	\end{equation}
	and
	\begin{align}\label{1eq: defn of F(1,2)}
	& F_{\rho, \shskip \varnu }^{(1)} (z) = {_2 F_1} \hskip - 2 pt \lp   \frac {\rho \hskip - 0.5 pt + \hskip - 0.5 pt \varnu } {2}; \frac {1 \hskip - 0.5 pt + \hskip - 0.5 pt \rho + \hskip - 0.5 pt \varnu  } {2}; 1 \hskip - 0.5 pt + \hskip - 0.5 pt\varnu  \shskip; z   \rp \hskip - 2 pt ,   
	F_{\rho, \shskip \varnu }^{(2)} (z) = {_2 F_1} \hskip - 2 pt \lp   \frac {\rho \hskip - 0.5 pt - \hskip - 0.5 pt \varnu } {2}; \frac {1 \hskip - 0.5 pt + \hskip - 0.5 pt \rho \hskip - 0.5 pt - \hskip - 0.5 pt \varnu  } {2}; 1 \hskip - 0.5 pt - \hskip - 0.5 pt \varnu  \shskip; z   \rp \hskip - 2 pt.
	\end{align}
	We have the identity
	\begin{equation}\label{0eq: general W-S, C}
	\begin{split}
	& \int_{0}^{2 \pi} \int_0^\infty \bfJ_{ \mu,\shskip  d}    \big(  x e^{i\phi} \big)  e (- 2 x y \cos (\phi + \theta) )  x^{2 \rho - 1}  d x \shskip d \phi
	=  \\
	&   \frac { C_{\mu, \shskip d}^{\shskip \rho}    F_{\rho, \shskip \mu+d}^{(1)} \big(4 / y^2 e^{2i \theta}\big) \hskip -1  pt F_{\rho, \shskip \mu-d}^{(1)} \big(4 e^{ 2i \theta} / y^2 \big) \hskip -1.5 pt } {y^{2\rho + 2 \mu} e^{2 i d \theta}} \hskip -1.5 pt + \hskip -1.5 pt \frac { C_{- \mu, \shskip - d}^{\shskip \rho} F_{\rho, \shskip \mu+d}^{(2)} \big(4 / y^2 e^{2i \theta}\big) \hskip -1  pt F_{\rho, \shskip \mu-d}^{(2)} \big(4 e^{2i \theta} / y^2 \big) \hskip -1.5 pt} {{ y^{2\rho -   2 \mu } e^{- 2i d \theta}}}
	\end{split}
	\end{equation}
	for $y \in [0, \infty)$ and $\theta \in [0, 2 \pi)${\rm;}  for $d = 0$,   the right hand side of the identity  is to be replaced by its limit if $\mu = 0$. Moreover, the identity {\rm\eqref{0eq: general W-S, C}} is valid under the weaker condition	$ |\Re \mu| < \Re \rho < 1 $ if we further assume that $ y > 2$. 
%	for either $y > 2$ or $0 < y \leqslant 2 $ but $\theta \neq 0, \pi$. 
\end{thm}

\begin{rem}
Since $F_{\rho, \shskip \varnu }^{(1)} ( 1/z)$ and $F_{\rho, \shskip \varnu }^{(2)} (1/z)$ are defined via analytic continuation within the unit circle $|z| < 1$, the formula {\rm\eqref{0eq: general W-S, C}} is not quite illuminating for $y < 2${\rm;} it is not so clear  {\it a priori} that the right hand side of  {\rm\eqref{0eq: general W-S, C}} is a well-defined function on the complex plane.  An alternative expression of {\rm\eqref{0eq: general W-S, C}} obtained from Proposition {\rm\ref{prop: hypergeometric |z| < 1}}  would be more transparent in terms of  the Gauss hypergeometric series for $ y < 2$. 
\end{rem}

By the Gaussian formula for ${_2F_1} (a, b; c; 1)$ (see \cite[\S 2.1]{MO-Formulas})
\begin{align*}
{_2F_1} (a, b; c; 1) = \frac {\Gamma (c) \Gamma (c-a-b)} {\Gamma (c-a) \Gamma (c-b)}, \hskip 15 pt \Re (a+b-c ) < 0, \, c \neq 0, -1, -2,...,
\end{align*}
together with the duplication  and the reflection formula for the gamma function, it is straightforward to derive the following corollary.

\begin{cor}\label{cor: d integer}
	For $ |\Re \mu| < \Re \rho < \bfrac 1 2 $, we have
	\begin{equation*}%\label{0eq: corollary}
	\begin{split}
	\int_{0}^{2 \pi} \hskip -1 pt \int_0^\infty & \bfJ_{ \mu,\shskip  d}    \big(  x e^{i\phi} \big)  e (- 4 x   \cos \phi )  x^{2 \rho - 1}  d x \shskip d \phi
	= \big(2 / \pi^3 \big) (8 \pi)^{-2\rho} \cos (\pi \rho) \Gamma (1/2 - \rho)^2 \, \cdot \\
	& \hskip -5 pt  \sin (\pi (\rho + \mu)) \sin (\pi (\rho - \mu)) \Gamma (\rho+\mu+d)  \Gamma (\rho+\mu-d) \Gamma (\rho-\mu+d)  \Gamma (\rho-\mu-d) .
	\end{split}
	\end{equation*}
\end{cor}

We shall prove Theorem \ref{thm: general W-S formula} by exploiting a soft   method that combines  asymptotic analysis of oscillatory integrals and a uniqueness result for differential equations. Precisely, it will be shown that the two sides of \eqref{0eq: general W-S, C} satisfy the same asymptotic as $y \ra \infty$ and also the same (hypergeometric) differential equations and hence are forced to be equal. 

This method of proof grew out of the author's previous work \cite{Qi-II-G} on a similar-looking integral which, in the notation of this paper, may be written as follows
\begin{align}\label{1eq: integral Fourier sqrt}
  \int_{0}^{2 \pi} \int_0^\infty \bfJ_{ \mu,\shskip  d}    \big(     x^{\frac 1 2} e^{\frac 1 2 i\phi} \big)  e (- 2 x y \cos (\phi + \theta) )    d x \shskip d \phi.
\end{align}
This Fourier-transform integral is however entirely different in nature. Roughly speaking, it has an explicit formula in terms of the Bessel function of halved order  $\bfJ_{ \frac 1 2 \mu,\shskip  \frac 1 2d} \big(1/4 y e^{i \theta} \big)$ (see \cite{Qi-II-G} for the formula (in slightly different notation)). That formula was used to extend the Waldspurger formula of Baruch and Mao from totally real to arbitrary number fields; see \cite{BaruchMao-NA,BaruchMao-Real,BaruchMao-Global} and \cite{Chai-Qi-Bessel,Chai-Qi-Wald}. The $  x^{\frac 1 2} e^{\frac 1 2 i\phi}$ above in \eqref{1eq: integral Fourier sqrt} is {\it squared} into the $x e^{i \phi}$  in \eqref{0eq: general W-S, C} (and also in \eqref{0eq: special W-S, C}), since it is the {\it symmetric square lift}  $\mathrm{Sym}^2$ under consideration.

Any direct method seems impossible as there is no formula available in the literature to  deal with the radial integration---the most outstanding obstacle in the extension of integral formulae from real to complex numbers. It is extremely lucky that the formula for \eqref{1eq: integral Fourier sqrt} can be proven by known formulae in the spherical case ($d = 0$); see \cite{Qi-Sph}. For the integral under consideration, however, even the spherical case is inaccessible by direct method.

\vskip 5 pt

\subsubsection*{Special Weber-Schafheitlin formula after regularization}

Define 
\begin{equation}\label{0eq: defn P mu m (z)}
P_{ \mu,\shskip  d} (z) = \frac {(z/2)^{\mu + d} (\widebar z/2)^{\mu - d}} {\Gamma  ( \mu + d + 1 ) \Gamma  ( \mu - d + 1 )},
\end{equation}
and
\begin{equation}\label{0eq: defn of R}
\boldsymbol {R}_{ \mu,\shskip  d} (z) =  
\frac {1} {\sin (\pi \mu)} \lp P_{-\mu,\shskip  -d} (4 \pi   z) -  P_{ \mu,\shskip  d} (4 \pi   z) \rp.
\end{equation} 
Again,  the right hand side is replaced by its limit when  $ \mu$ is an integer.
We now state the regularized special Weber-Schafheitlin integral formula for $\bfJ_{ \mu,\shskip  d} (z)$ as follows.

\begin{thm}\label{thm: regularized W-S, C}
	 Suppose that $|\Re \mu | < \bfrac 1 2 $.
	 Define the regularized Bessel function 
	 \begin{align}\label{0eq: defn M}
	 \boldsymbol{M}_{ \mu,\shskip  d} (z) = \boldsymbol{J}_{ \mu,\shskip  d} (z) - \boldsymbol{R}_{ \mu,\shskip  d} (z).
	 \end{align}
	 Let
	 \begin{align}
	  Y (z) =    \frac { \left| z + \sqrt {z^2 - 4}  \right| } { 2  }  , \hskip 10 pt E (z) = \frac {  z + \sqrt {z^2 - 4}   } { \left|z + \sqrt {z^2 - 4} \right|  } . 
	 \end{align}
	 We have
	 \begin{equation}\label{0eq: special W-S, C}
	 \begin{split}
	  \int_{0}^{2 \pi} \hskip - 1 pt   \int_0^\infty  & \boldsymbol{M}_{ \mu,\shskip  d }  \big(  x e^{i\phi} \big)  e (- 2 x y \cos (\phi + \theta) )  \frac {  d x \shskip d \phi} {x}
	 =  \frac {1} {d^2 - \mu^2} \cdot \\
	 &   \left\{ \hskip -1 pt  \lp  Y \big(y e^{i \theta} \big)^{ 2 \mu } E \big(y e^{ i \theta} \big)^{2 d} - y^{2 \mu } e^{ 2 i d \theta}  \rp + \lp \frac 1 { Y \big(y e^{i \theta} \big)^{ 2 \mu } E \big(y e^{ i \theta} \big)^{2 d}} - \frac 1  {y^{2 \mu  } e^{ 2 i d \theta}}  \rp 
	 \hskip -1 pt  \right\} \hskip -2 pt 
	 \end{split}
	 \end{equation}
	 for $y \in (0, \infty)$ and $\theta \in [0, 2 \pi)${\rm;} for $d = 0$,   the right hand side of the identity  is to be replaced by its limit if $\mu = 0$.
\end{thm}

%For the application to Beyond Endoscopy,  Theorem \ref{thm: regularized W-S, C} is restated in the spherical case as follows.

 %the work of Venkatesh considers only spherical Bessel functions. 

%Note that in the polar coordinates
%\begin{equation}\label{0def: P mu m (x eiphi)}
%P_{ \mu,\shskip  d}  \lp x e^{i\phi} \rp = \frac {(x/2)^{4 \mu } e^{i m \phi}  } {\Gamma  ( 1 - 2\mu -m/2  ) \Gamma  ( 1 - 2\mu + m/2  )}.
%\end{equation}

\subsection{Application to Beyond Endoscopy over arbitrary number fields}

In the  work of Venkatesh \cite{Venkatesh-BeyondEndoscopy},  Langlands' proposal of Beyond Endoscopy \cite{Langlands-BE} is executed for the symmetric square lift $\mathrm{Sym}^2$ on $\GL_2$, giving the classification of dihedral forms---forms whose symmetric square has a pole. It is proven that dihedral forms  correspond to Gr\"ossen-characters of
quadratic field extensions. This result is originally due  to Labesse and Langlands \cite{LL-Endoscopy} by endoscopic methods.

The fundamental tool used by Venkatesh at the beginning is the Kuznetsov relative trace formula of Bruggeman and
Miatello. Poisson summation is then applied at the stage after Kuznetsov-Bruggeman-Miatello. Afterwards, the local exponential sums are evaluated and units of the quadratic field extension enter into the analysis. The Archimedean theory of Venkatesh is contained in his Proposition 7, in which  arises naturally the Fourier transform of Bessel functions due to Kuznetsov-Bruggeman-Miatello and Poisson. 

Venkatesh works over a number field. %Since it is over a number field,
As such, the method is extremely notationally complicated in his \S 4. He however takes great care in guiding the reader by illustrating the main ideas over $\BQ$ in his \S 3. More details may be found in his thesis \cite{Venkatesh-Thesis}.

The main results of Venkatesh (in his \S 4) are stated in the setting of a totally real number field. The only serious obstacle in
the general case (involving complex places) is the validity of a certain integral transform as in his Proposition 7. Except for this, his method clearly does not rely on the totally real assumption  in any important way.

In the present paper, Proposition 7 of Venkatesh is extended to complex places so that his main results are generalized to arbitrary number fields.

%\red{More words and references here.}

For other works on Beyond Endoscopy, see for example \cite{Sarnak-BE,Herman-11,Herman-12,Herman-Asai-16,FLN,YS-1-13,YS-2-17,White-14,Altug-1-15,Altug-2,Altug-3}. 

\vskip 5 pt

\subsubsection*{Notation}

Let $F$ be an (arbitrary) number field. Let $\varv$  stand for a place of $F$ and $F_{\varv}$ denote
the completion of $F$ at $\varv$.  %We are only concerned with Archimedean places $\varv | \infty$. 
Let $S_{\infty}$ be the set of Archimedean places. Write $\varv | \infty$ as the abbreviation for $\varv \in S_{\infty}$.   %Let $|\ |_{\varv}$ denote the standard module on $F_{\varv}$.  
Define $F_{\infty}$   to be $F \otimes \BR = \prod_{ \varv | \infty } F_{\varv}$. For $x \in F_{\infty}$ let $\mathrm{Norm} (x)$ denote the norm of $x$.  We fix the Haar measure on $F_{\infty}$  corresponding
to the product of $d x$ at real places and $|d x \hskip -1 pt \wedge \hskip -1 pt d \widebar x|$ at complex places. Let $\psi_{\infty}: F_{\infty} \ra \BC$ be the additive character $\psi_{\infty} (x) = e \big(\Tr_{F_{\infty} /\BR } (x) \big)$.  

Let $\fra$ be the vector space $\BR^{|S_{\infty}|}$. Accordingly, we denote a typical element by
$ t = (t_{\varv})_{ \varv | \infty }$. Let $\fra_{\BC} $ be its complexification. Let $d t$ be the usual Lebesgue
measure on $\fra$. Following \cite{BM-Kuz-Spherical}, we also equip $\fra$ with the positive measure $d \mu (t) =  \prod_{ \varv | \infty } \hskip -1 pt \bfrac 1 2 t_{\varv} \sinh (\pi t_{\varv} ) d t_{\varv} .$ Let   ${\mathrm{Pl}} : \fra \ra \BC $ be defined by
\begin{align}\label{1eq: defn of Pl}
{\mathrm{Pl}} (t) = \prod_{\varv \text{ real} }  2 \cosh (\pi t_{\varv}) \prod_{\varv \text{ complex} }   { 2 \sinh (\pi t_{\varv})} / {t_{\varv}} . 
\end{align}
\footnote{In \S 6.7.2 and 6.7.3 of \cite{Venkatesh-BeyondEndoscopy}, the $1/2$ in (106) and (112) should be removed, so we have $ 2 \cosh (\pi t_{\varv})$ here instead of $ \cosh (\pi t_{\varv})$.} Note that $d \mu (t) / \mathrm{Pl} (t) $ is the Plancherel measure on the set of spherical tempered representations of $\PGL_2 (F_{\infty})$.
Define the logarithm function $ \log_F : F_{\infty}^{\times} \ra \fra $ by
\begin{align}\label{1eq: log}
\log_F (x) = (\log |x_{\varv}| )_{\varv | \infty}  .
\end{align}

\vskip 5 pt

\subsubsection*{Bessel kernel and Bessel transform}

Let $M > 2$, $N > 6$. We set
$ \SH (M, N) $ to be the space of functions $h : \fra \ra \BC $ that are of the following form,
\begin{align*}
h (t) = \prod_{\varv | \infty} h_{\varv} (t_{\varv}),
\end{align*}
where each $h_{\varv} : \BR \ra \BC $ extends to an even holomorphic function on the strip
$\big\{ s  = t + i \sigma : |\Im s | \leqslant M \big\}$ such that,  on the horizontal line $\Im s = \sigma$ ($|\sigma| \leqslant M$), we have uniformly 
\begin{align*}
h_{\varv} (t + i \sigma) \Lt e^{-\pi |t|} (|t|+1)^{- N}.
\end{align*}

We define the Bessel kernel $\SB   : F_{\infty}^{\times} \times \fra_{\BC} \ra \BC $  as follows. For $x \in F_{\infty}^{\times}$ and $\varnu  \in \fra_{\BC} = \BC^{|S_{\infty}|}$,
\begin{align*}
 \SB (x, \varnu ) = \prod_{ \varv | \infty } B_{\varv} (x_{\varv}, \varnu _{\varv}),
\end{align*}
where
\begin{align*}
& B_{\varv} (x_{\varv}, \varnu _{\varv}) = \frac 1 {\sin (\pi \varnu _{\varv}) } \big( J_{-2 \varnu _{\varv}} (4 \pi \sqrt {x_{\varv}}) - J_{2 \varnu _{\varv}} (4 \pi \sqrt {x_{\varv}}) \big), \\
& B_{\varv} (-x_{\varv}, \varnu _{\varv}) = \frac 1 {\sin (\pi \varnu _{\varv}) } \big( I_{-2 \varnu _{\varv}} (4 \pi \sqrt {x_{\varv}}) - I_{2 \varnu _{\varv}} (4 \pi \sqrt {x_{\varv}}) \big) = \frac {4 \cos (\pi \varnu _{\varv})} \pi  K_{2 \varnu _{\varv}} (4 \pi \sqrt {x_{\varv}}) ,
\end{align*}
if $\varv$ is real and $x_{\varv} > 0$, and
\begin{equation*}
B_{\varv} (x_{\varv}, \varnu _{\varv}) =  \frac 1 {\sin (\pi \varnu _{\varv}) } \big( J_{- \varnu _{\varv}} (4 \pi \sqrt {x_{\varv}}) J_{- \varnu _{\varv}} (4 \pi \sqrt { \widebar x_{\varv}}) - J_{ \varnu _{\varv}} (4 \pi \sqrt {x_{\varv}}) J_{ \varnu _{\varv}} (4 \pi \sqrt { \widebar x_{\varv}}) \big),  
\end{equation*}
if $\varv$ is complex; by definition,  we have
\begin{align*}%\label{1eq: B = J}
 B_{\varv} ( x_{\varv}, \varnu _{\varv}) = \bfJ_{ \varnu _{\varv} } (\hskip -2 pt \sqrt {x_{\varv}}).
\end{align*}

\begin{rem}
	  For complex $\varv$, our definition of $B_{\varv} (x_{\varv}, \varnu _{\varv})$ is slightly different from that of \cite{Venkatesh-BeyondEndoscopy}, in which he uses $I_{\varnu } (x)$ instead of $J_{\varnu } (x)$. But we have the formulae $ I_{\varnu } (x) = e^{- \frac 1 2 \pi i \varnu } J_{\varnu }  ( i x  ) $ and  $ I_{\varnu } (x) = e^{ \frac 1 2 \pi i \varnu } J_{\varnu }  ( - i x  ) $ {\rm(}see \cite[3.7 (2)]{Watson} and it is understood  that $i = e^{\frac 1 2 \pi i }$ and $ - i = e^{- \frac 1 2 \pi i }${\rm)}. So the difference is only up to the sign of $x_{\varv}$. This sign difference however would cause a little inconsistency between the Kuznetsov formula for $\SL_2 (\BC)$ in \cite{M-W-Kuz}{\rm(}or \cite{BM-Kuz-Spherical}{\rm)} and that in \cite{B-Mo}. From various sources in the literature, it is suggested that the latter {\rm(}and hence the formula of $B_{\varv} (x_{\varv}, \varnu _{\varv})$ here in terms of $ J_{\varnu } (x) ${\rm)} should be the correct one.
\end{rem}

Let $h (t)$ be a test function on $\fra$ belonging to $\SH (M, N)$ and define its Bessel integral transform $\varphi : F_{\infty}^{\times} \ra \BC $ by 
\begin{align}\label{1eq: Bessel transform}
\varphi (x) = \int_{ \fra } h (t) \SB (x, it) d \mu (t). 
\end{align} 
This Bessel transform arises in  the Kuznetsov-Bruggeman-Miatello trace formula as its Archimedean component.

\vskip 5 pt

\subsubsection*{Proposition 7 of Venkatesh}

The following theorem was proven for totally real $F$ by Venkatesh \cite[Proposition 7]{Venkatesh-BeyondEndoscopy}, as the main ingredient in the Archimedean theory of his work. He mentioned that  at the time he had not managed to accomplish the general case. %He also indicated that the extension to complex places would only require verifying the validity of a certain explicit integral transform for the Bessel functions over complex numbers. 
Given this theorem, the main results of Venkatesh for $\mathrm{Sym}^2$ may be now extended to arbitrary number field $F$.

\begin{thm}[Proposition 7 of Venkatesh]\label{thm: Venkatesh} Let $F$ be a number field. Suppose that $h (t)  \in \SH (M, N)$ and $\varphi (x)$ is the Bessel transform of $h (t)$ defined by \eqref{1eq: Bessel transform}. Let $\mathrm{Pl} : \fra \ra \BC $ be defined as in {\rm\eqref{1eq: defn of Pl}}. Define $ \widehat { h \cdot \mathrm{Pl}  } : \fra \ra \BC $ to be the Fourier transform of $h (t) \mathrm{Pl} (t)$, so for $\tv \in \fra$,
	\begin{align}\label{1eq: hat of h Pl}
\widehat { h \cdot \mathrm{Pl} } (\tv) = \int_{ \fra } e^{\shskip i \shskip \langle t, \tv \rangle } h (t) \mathrm{Pl} (t)  d t ,
	\end{align}
	where $\langle t, \tv \rangle = \sum_{ \varv | \infty } t_{\varv} \tv_{\varv} $. Defining $\log_F : F_{\infty}^{\times} \ra \fra $ as in \eqref{1eq: log}, we have
	\begin{align}\label{1eq: main identity}
	 \widehat { h \cdot \mathrm{Pl} } (\log_F |\kappa| ) = \int_{ F_{\infty} } \varphi \lp \frac 1 {2 + \kappa+ \kappa\-} x^2 \rp \psi_{\infty} (x) \frac {d x} {\mathrm{Norm} (x)};
	\end{align}
the	convergence is guaranteed  as long as $M, N$ are sufficiently large.
\end{thm}

Since the situation in the presence of multiple Archimedean places is a ``product" of
situations involving just one Archimedean place,  it suffices to check it in the case that $F_{\infty} = \BR$ or $F_{\infty} = \BC$. The case $F_{\infty} = \BR$ has already been settled by Venkatesh (see also \S \ref{sec: analysis over R}). The case $F_{\infty} = \BC$ will be proven in \S \ref{sec: Proof of Venkatesh} by using our regularized Weber-Schafheitlin integral formula for $ \bfJ_{ \mu} (z) $ in Theorem \ref{thm: regularized W-S, C}. %As alluded to before, the regularized Weber-Schafheitlin integral formulae for $B_{\varnu } (x)$ and $K_{\varnu } (x)$ would enable us to give a unified treatment for both  

\begin{cor}
	The main results of Venkatesh on Beyond Endoscopy for $\mathrm{Sym}^2$ on $\GL_2$, in particular his Proposition {\rm 2} and Theorem {\rm 1}, are valid over an arbitrary number field $F$.  
\end{cor}

\begin{acknowledgement}
I thank all the participants of my analytic number theory seminar, especially, Dongwen Liu and Zhicheng Wang, at Zhejiang University  in the autumn of 2018---Venkatesh's paper was the first one we studied in the seminar. % because of its analytic number theoretic techniques and representation theoretic background. 
I thank Roman Holowinsky and Akshay Venkatesh for their comments. I also thank the referees for careful
readings and helpful comments.
\end{acknowledgement}

	\section{Preliminaries}

	\subsection{Classical Bessel functions}\ 
	
	\vskip 5 pt
	
	\subsubsection{Basic properties of  $J_{\varnu } (z)$, $H^{(1 )}_{\varnu }  (z) $ and $H^{(2)}_{\varnu }  (z) $} Let $\varnu  $ be a complex number.
	Let $J_{\varnu } (z)$, $H^{(1,\shskip  2)}_{\varnu }  (z) $ denote the   Bessel function of the first kind and the Hankel functions of order $\varnu $. They all satisfy the Bessel differential equation
	\begin{equation}\label{eq: Bessel Equation}
	z^2 \frac {d^2 \tw} {d z^2}  (z) + z \frac {d \tw} {d z} (z) + \lp z^2 - \varnu ^2 \rp  \tw(z) = 0.
	\end{equation}
	%and $I_{\varnu }  (z) $ the modified Bessel function of the first kind of order $\varnu $. 
	The function $J_{\varnu } (z)$ is defined by the series    (see \cite[3.1 (8)]{Watson})
	\begin{equation}\label{2def: series expansion of J}
	J_{\varnu } (z) = \sum_{n=0}^\infty \frac {(-)^n \lp   z /2 \rp^{\varnu +2n } } {n! \Gamma (\varnu  + n + 1) }.
	\end{equation}
	When $\varnu  \neq -1, -2, -3, ...$, it follows from \cite[3.13 (1)]{Watson} that
	\begin{equation}\label{2eq: bound for J}
	  J_{\varnu } (z)   = \frac {(z/2)^{\varnu }} {\Gamma (\varnu +1)} \lp 1 + O_{\varnu } \lp  |z|^2 \rp \rp , \hskip 10 pt |z| \Lt 1,
	\end{equation}
	and, together with  the Bessel differential equation and the recurrence formula \cite[3.2 (4)]{Watson},
	\begin{align*}
	z J_{\varnu }' (z) = \varnu  J_{\varnu } (z) - z J_{\varnu +1} (z),
	\end{align*} that
	\begin{equation}\label{2eq: bound for J, 2}
	 z^{r} (d/dz)^{r} J_{\varnu } (z) \Lt_{\, r, \shskip \varnu } \left|z^{\varnu } \right| , \hskip 10 pt |z| \Lt 1,
	\end{equation}
	with the implied constants uniformly bounded when $\varnu $ lies in a given compact subset of $\BC \smallsetminus \{-1, -2, -3, ...\} $.

	We have the following connection formulae  (see \cite[3.61 (1, 2)]{Watson})
	\begin{align}
	\label{2eq: J and H} 
	& J_\varnu  (z) = \frac {H_\varnu ^{(1)} (z) + H_\varnu ^{(2)} (z)} 2, \hskip 30 pt 
	J_{-\varnu } (z) =  \frac {e^{\pi i \varnu } H_\varnu ^{(1)} (z) + e^{-\pi i \varnu } H_\varnu ^{(2)} (z) } 2.
	\end{align}
	
	According to \cite[7.2 (1, 2)]{Watson} and \cite[7.13.1]{Olver}, we have  Hankel's expansion of $H^{(1)}_{\varnu } (z)$ and $H^{(2)}_{\varnu } (z)$ as follows,
	\begin{align}\label{2eq: asymptotic H (1)}
	H^{(1)}_{\varnu } (z) &= \lp \frac 2 {\pi z} \rp^{\frac 1 2} e^{ i \lp z - \frac 12 {\pi \varnu }    - \frac 14 \pi    \rp }    
	\lp \sum_{n=0}^{N-1} \frac {(-)^n \cdot (\varnu , n) } {(2iz)^n} + E^{(1)}_N (z) \rp, \\
	\label{2eq: asymptotic H (2)}
	H^{(2)}_{\varnu } (z) &= \lp \frac 2 {\pi z} \rp^{\frac 1 2} e^{ - i \lp z - \frac 12 {\pi \varnu }    - \frac 1 4 \pi   \rp }   
	\lp \sum_{n=0}^{N-1} \frac {  (\varnu , n) } {(2iz)^n} + E^{(2)}_N (z) \rp,
	\end{align} 
	with $(\varnu , n) = \Gamma \lp \varnu  + n +   1 / 2 \rp / n! \Gamma \lp \varnu  - n +   1 / 2 \rp$, of which \eqref{2eq: asymptotic H (1)} is valid when $z$ is such that $- \pi +  \delta \leqslant \arg z \leqslant 2 \pi -   \delta$,  and  \eqref{2eq: asymptotic H (2)} when $- 2 \pi +  \delta \leqslant \arg z \leqslant   \pi -  \delta$,   $  \delta  $ being any positive acute angle, and
	\begin{align}\label{2eq: estimates for E}
	z^r (d/dz)^r E_N^{(1, \shskip 2)} (z)  \Lt_{\shskip \delta, \shskip r, \shskip N, \shskip \varnu } 1/ |z|^{N }
	\end{align} 
	for   $|z| \Gt 1$ and $\arg z$  in the range indicated as above. In view of the error bounds in \cite[7.13.1]{Olver}, the dependence of the implied constant on $\varnu $ is uniform in any given compact set.
	
	\vskip 5 pt

	\subsubsection{Some integral formulae}
	
	When the order $\varnu  = m$ is an integer, we have the integral representation of Bessel for $J_m (z)$ as follows (see \cite[2.2 (1)]{Watson}),
	\begin{align}\label{2eq: integral repn of J}
	i^{\shskip m} J_{m} (z) = (- i)^m J_{- m} (z) 
	= \frac { 1 } {2 \pi  } \int_0^{2\pi} e^{ i z \cos \phi - i m \phi} d \phi .
	\end{align}
	We have the following formula due to Weber, Sonine and Schafheitlin (see \cite[13.24 (1)]{Watson} and \cite[6.561 14]{G-R}),
	\begin{align}\label{1eq: WSS formula}
	\int_0^{\infty} \frac {J_{\varnu } (x) d x} {x^{\varnu - \mu+1}} = \frac {\Gamma (\mu / 2)} {2^{\varnu -\mu+1} \Gamma (\varnu  - \mu/2 +1)},
	\end{align}
	in which $0 < \Re \mu < \Re \varnu  + \bfrac{3}{\shskip 2}$ (this is the domain of convergence in \cite[6.561 14]{G-R}, while it is literally $0 < \Re \mu < \Re \varnu  + \bfrac{1}{\shskip 2}$ in \cite[13.24]{Watson}, for Watson only considers the domain of absolute convergence).

	\subsection{\texorpdfstring{Preliminaries on the Bessel function $\bfJ_{ \mu,\shskip  d} (z)$}{Preliminaries on the Bessel function $J_{ \mu,\shskip  d} (z)$}}

	Replacing $d/dz$ by $\partial / \partial z$, we denote by $\nabla_{\varnu }$ the differential operator that occurs in \eqref{eq: Bessel Equation}, namely,
	\begin{equation}\label{2eq: nabla}
	\nabla_{\varnu } = z^2 \frac {\partial^2 } {\partial z^2}  + z \frac {\partial  } {\partial z} +   z^2 - \varnu ^2    .
	\end{equation}
	Its conjugation will be  denoted by $\overline \nabla_{\varnu }$,
	\begin{equation}\label{2eq: nabla bar}
	\ov  \nabla_{\varnu } = \widebar z^2 \frac {\partial^2 } {\partial \widebar z^2}  + \widebar z \frac {\partial  } {\partial \widebar z} +  \widebar z^2 - \varnu ^2    .
	\end{equation}
	From the definition of $\bfJ_{ \mu,\shskip  d} (z) $ as in (\ref{0def: J mu m (z)}, \ref{0eq: defn of Bessel}), we infer that 
	\begin{align}\label{2eq: nabla J = 0}
	\nabla_{\mu + d} \lp \bfJ_{ \mu,\shskip  d} (z/4 \pi) \rp = 0, \hskip 10 pt \ov \nabla_{  \mu - d} \lp \bfJ_{ \mu,\shskip  d} (z/4 \pi) \rp = 0.
	\end{align} 
	
	Recall the definition of $\boldsymbol {R}_{ \mu,\shskip  d} (z)$ given by (\ref{0eq: defn P mu m (z)}, \ref{0eq: defn of R}). Suppose at the moment that $|z| \leqslant 2$, say. It follows from \eqref{2eq: bound for J} that if $\mu$ is not an integer, then
	\begin{align}\label{2eq: bound for J mu m}
	\bfJ_{ \mu,\shskip  d} (z) - \boldsymbol {R}_{ \mu,\shskip  d} (z)  \Lt_{\, \mu,\shskip  d} \left| |z|^{2- 2\mu} \right| + \left| |z|^{2 +2 \mu} \right|.
	\end{align} 
	Some calculations by the formulae  of $ \left.( \partial J_{\varnu }  (z) /\partial \varnu  ) \right|_{\varnu  = \pm m}$  ($m = 0, 1, 2,...$) in \cite[\S 3.52 (1, 2)]{Watson} would imply that in the non-generic case when $  \mu  $ is an integer, %we have
	\begin{align}\label{2eq: bound for J mu m, 2}
	\bfJ_{ \mu,\shskip  d} (z) - \boldsymbol {R}_{ \mu,\shskip  d} (z) \Lt_{\, \mu,\shskip  d}   |z|^{2 - 2 |\mu|} \log (4/|z|).
	\end{align}
	It will be convenient to unify \eqref{2eq: bound for J mu m} and \eqref{2eq: bound for J mu m, 2} in a slightly weaker form as follows,
	\begin{align}\label{2eq: bound for J mu m, weak}
	\bfJ_{ \mu,\shskip  d} (z) - \boldsymbol {R}_{ \mu,\shskip  d} (z)  \Lt_{\, \mu,\shskip  d, \shskip \lambdaup}  |z|^{2 - 2 \lambdaup},
	\end{align}
	with $\lambdaup = |\Re \mu|$ if $  \mu$ is not an integer and  $\lambdaup > |\Re \mu|$ if otherwise. Further, we have
	\begin{align}\label{2eq: bound for J mu m, weak, 2}
	z^{\shskip r} \widebar z^{ \,s} (\partial /\partial z)^r (\partial / \partial \widebar z)^{s} \bfJ_{ \mu,\shskip  d} (z)   \Lt_{\, r,\shskip s, \shskip \mu,\shskip  d, \shskip \lambdaup}  |z|^{- 2 \lambdaup} .
	\end{align}
	For example,  in the generic case, this is a direct consequence of  the bounds in \eqref{2eq: bound for J, 2}.

	\delete{Moreover, the bound in \eqref{2eq: bound for J, 2} implies that in the generic case 
	\begin{align}\label{2eq: bound for J mu m, 3}
	 z^{\shskip\alpha} \widebar z^{ \,s} (\partial /\partial z)^\alpha (\partial / \partial \widebar z)^{s} \bfJ_{ \mu,\shskip  d} (z)   \Lt_{\, \alpha,\shskip s, \shskip \mu,\shskip  d} \left| |z|^{- 2 \mu} \right| + \left| |z|^{ 2 \mu} \right|, \hskip 10 pt |z| \Lt 1;
	\end{align}
the estimate in \eqref{2eq: bound for J mu m, 3} is almost valid in the non-generic case at the cost of $\log (1/|z|)$, namely,
\begin{align}\label{2eq: bound for J mu m, 4}
z^{\shskip\alpha} \widebar z^{ \,s} (\partial /\partial z)^\alpha (\partial / \partial \widebar z)^{s} \bfJ_{ \mu,\shskip  d} (z)   \Lt_{\, \alpha,\shskip s, \shskip \mu,\shskip  d} |z|^{2 - 2 |\mu|} (|\log (1/|z|)| + 1), \hskip 10 pt |z| \Lt 1.
\end{align}
}
	
	In view of the connection formulae in \eqref{2eq: J and H}, we have another expression of $ \bfJ_{ \mu,\shskip  d} (z)$ in terms of Hankel functions,
	\begin{equation}\label{2eq: J = H1 + H2}
	\bfJ_{ \mu,\shskip  d} (z) =  \frac i  2 \hskip -1 pt \left( e^{  \pi i \mu} H^{(1)}_{\mu,\shskip  d} \lp 4 \pi   z \rp  - e^{-   \pi i \mu} H^{(2)}_{\mu,\shskip  d} ( 4 \pi   z ) \right),
	\end{equation}
	with the definition
	\begin{equation}\label{0def: H mu m (z)}
	H^{(1,\shskip  2)}_{\mu,\shskip  d} (z) = H^{(1,\shskip  2)}_{\mu + d} \lp   z \rp  H^{(1,\shskip  2)}_{\mu - d} \lp  { \widebar z} \rp.
	\end{equation}
	The reader should be warned that the product in \eqref{0def: H mu m (z)} is {\it not well-defined} as function on $\BC \smallsetminus \{0\}$.
	By \eqref{2eq: asymptotic H (1)}-\eqref{2eq: estimates for E}, we may write
		\begin{align}\label{1eq: J = W + W}
		\boldJ_{ \mu,\shskip  d} (z) =   {e (4 \Re  z)}   \boldsymbol{W}   (z) +   {e (- 4 \Re z)}   \boldsymbol{W}   (- z) + \boldsymbol{E}_N  (z), \
		\end{align}
		where 
		$\boldsymbol{W} (z)$ and $ \boldsymbol{E}_N (z) $ are real analytic functions on $\BC \smallsetminus \{0\}$ satisfying 
		\begin{align}\label{1eq: derivatives of W}
		z^{  r} \shskip \widebar z^{ \,s} (\partial /\partial z)^ r (\partial / \partial \widebar z)^{s} \boldsymbol{W}  (z) \Lt_{\,  r,\shskip s, \shskip N, \shskip \mu,\shskip  d} 1 / |z|,   & \\
		\label{1eq: derivatives of E}
		(\partial /\partial z)^ r (\partial / \partial \widebar z)^{s} \boldsymbol{E}_N   (z) \Lt_{\,  r,\shskip s, \shskip N, \shskip \mu,\shskip  d} 1 / |z|^{N+1},  &
		\end{align}   
	for $|z|  \geqslant 1$.

%	\red{is of exponential growth in certain radial directions.} 

	\subsection{The hypergeometric function} Our reference for the hypergeometric function is Chapter II of \cite{MO-Formulas}. 
	
	The hypergeometric function ${_2F_1} (a, b; c; z)$  is defined by the Gauss series
	\begin{align}
 {_2F_1} (a, b; c; z) = \frac {\Gamma (c)} {\Gamma (a) \Gamma (b)} \sum_{n=0}^{\infty} \frac {\Gamma (a+n) \Gamma (b+n)} {\Gamma (c+n) n!} z^n
	\end{align} 
	within its circle of convergence $|z| < 1$, and  by analytic continuation elsewhere. The series is absolutely convergent on the unit circle $|z| = 1$ if $\Re(a+b-c) < 0$. 
	The function $ {_2F_1} (a, b; c; z) $ is a single-valued analytic function of $z$ in the complex plane with a branch cut along the positive real axis from
	$1$ to $\infty$. Moreover, $  {_2F_1} (a, b; c; z)$ is analytic in $a$, $b$ and $c$, except for $c = 0, -1, -2,...$.
	
	The hypergeometric differential equation satisfied by $ {_2F_1} (a, b; c; z)$ is as follows,
	\begin{align}\label{2eq: hypergeometric equation}
	z (1-z) \frac {d^2 \tw} {d z^2}  (z) + (c - (a+b+1) z ) \frac {d \tw} {d z} (z) - a b  \tw(z) = 0.
	\end{align}
	It has three regular singular
	points $z = 0, 1, \infty$. In the generic case when none of $c$, $ a-b$ and $ c-a-b$ is an integer, two linearly independent solutions of \eqref{2eq: hypergeometric equation} in the vicinity of $z = \infty$ are given by
	\begin{equation}\label{2eq: G1 and G2}
	\begin{split}
	& G_1 (a, b; c; z) = z^{-a} {_2F_1} (a, a-c+1; a-b+1; 1/z), \\
	& G_2 (a, b; c; z) = z^{-b} {_2F_1} (b, b-c+1; b-a+1; 1/z). 
	\end{split}
	\end{equation}
	Finally, we record here  the transformation formula with respect to $z \ra 1/z$ (see \cite[\S 2.4.1]{MO-Formulas}),
	\begin{equation}\label{2eq: transformation z - 1/z}
	\begin{split}
	{_2F_1} (a, b; c; z)  =\, & \frac {\Gamma (c) \Gamma (b-a)} {\Gamma (b) \Gamma (c-a) } (-z)^{-a} {_2F_1} ( a, a-c+1; a-b+1; 1/z ) \\
	+\, & \frac {\Gamma (c) \Gamma (a-b)} {\Gamma (a) \Gamma (c-b) } (-z)^{-b} {_2F_1} ( b, b-c+1; b-a+1; 1/z ),\\
	&\hskip 40 pt |\arg (-z)| < \pi, \ a-b \neq \pm m, \ m=0, 1, 2,....
	\end{split}
	\end{equation}

	\subsection{A lemma on oscillatory integrals}
	
	The following lemma  will be very useful in our later analysis.
	
	\begin{lem}\label{lem: integration by parts}
		Let $a,  b$ and $\theta$ be real numbers such that $ |b| < |a| $. %Let $\theta \in [0, 2 \pi)$.  
		Let $0 < A < B \leqslant \infty$, $\gamma < \bfrac{1}{2}$, and $M$ be a positive integer.
Suppose that $ f (z) $ is a smooth function on $\BC$ which is supported on the annulus $  |z| \in [A, B]  $  and satisfies $$ x^{r} (\partial / \partial x)^{r} (\partial / \partial \phi)^{s} f  (x e^{i\phi}  ) \Lt_{r, \, s} x^{2 \gamma - 2} $$ for all nonnegative integers $r, s$ such that $r + s \leqslant 2 M$ {\rm(}equivalently, $$z^{  r} \shskip \widebar z^{ s} (\partial /\partial z)^r (\partial / \partial \widebar z)^{s} f (z) \Lt_{r, \, s} |z|^{2 \gamma - 2} $$ for all  $r + s \leqslant 2 M${\rm)}. Define
\begin{align*}
I (a, b, \theta ) = \int_{0}^{2 \pi} \int_0^{\infty} f \big(x e^{i\phi} \big) e ( x (a \cos (\phi + \theta) + b \cos \phi)  ) x d x \shskip d \phi.
\end{align*}
Then the integral $I (a, b, \theta )$ is convergent after integration by parts {\rm(}it is already absolutely convergent if $\gamma < 0${\rm)} and 
\begin{align*}
I  (a, b, \theta ) \Lt_{M, \, \gamma} \frac 1 {A^{2 M - 2\gamma} (|a| - |b|)^{2 M} } ,
\end{align*}
with the convergence and the implied constant  uniform for $\gamma$ in a compact set. Moreover, the integral $I  (a, b, \theta )$ gives rise to a continuous function in $(a, \theta)$ {\rm(}the continuity extends to all real values of $a $ if $ \gamma < 0 ${\rm)}.  
	\end{lem}
	
	\begin{proof}
	Put $p (x, \phi; a, b, \theta) = x (a \cos (\phi + \theta) + b \cos \phi)$. Define the  differential operator $
	\mathrm{D} = (\partial / \partial x)^2 + (\partial / \partial x)/x + (\partial / \partial \phi)^2 / x^2 $
	so that $$\mathrm{D} \big(e ( p (x, \phi; a, b, \theta) )\big) = -4 \pi^2 (a^2 + b^2 + 2 ab \cos \theta) \cdot e ( p (x, \phi; a, b, \theta)  ).$$ Note that $\mathrm{D} $ is self-adjoint. In view of the conditions on $f (z)$, it follows from an application of the partial integration with respect to $\mathrm{D}$ that
	\begin{align*}
	I  (a, b, \theta ) = - \frac 1 {4 \pi^2 (a^2 + b^2 + 2 ab \cos \theta) }   \int_{0}^{2 \pi} \int_0^{\infty} \mathrm{D}\shskip f \big(x e^{i\phi} \big) e (  p (x, \phi; a, b, \theta) ) x d x \shskip d \phi.
	\end{align*} 
	Since $x \mathrm{D}  f (x e^{i\phi} ) \Lt x^{2 \gamma - 3}$, the integral above is absolutely convergent as $\gamma < \bfrac{1}{2}$. Repeating the partial integration above $M$ times and then bounding the resulting integral trivially, along with  $ a^2 + b^2 + 2 ab \cos \theta \geqslant (|a|-|b|)^2 $ and $x \mathrm{D}^M  f (x e^{i\phi} ) \Lt_{M} x^{2 \gamma - 2M - 1} $,  we obtain the estimates for $I  (a, b, \theta )$ in the lemma. Finally, the continuity of $ I  (a, b, \theta ) $ is obvious.
	\end{proof}

\section{More on the hypergeometric function}\label{sec: hypergeometric}

We shall be concerned with the hypergeometric differential equation for 
\begin{align}\label{2eq: a, b, c}
a = \frac {\rho + \varnu } 2, \hskip 10 pt b = \frac {\rho - \varnu } 2, \hskip 10 pt c = \frac 1 2.
\end{align}
%	in which $\varnu  = \mu \pm d$ such that \red{....}
Define the corresponding hypergeometric differential operator 
\begin{align}\label{2eq: nabla, 1}
\nabla_{\rho, \shskip \varnu } =  4 z (1-z) \frac {\partial^2  } {\partial z^2}   + 2( 1 - 2 (\rho + 1) z ) \frac {\partial } {\partial z}  -  \lp \rho^2 - \varnu ^2 \rp ,
\end{align}
and its conjugate
\begin{align}\label{2eq: nabla, 2}
\ov \nabla_{\rho, \shskip \varnu } =  4 \widebar z (1-\widebar z) \frac {\partial^2  } {\partial \widebar z^2}   + 2 ( 1 - 2 (\rho + 1) \widebar z ) \frac {\partial } {\partial \widebar z}  -  \lp \rho^2 - \varnu ^2 \rp .
\end{align}
For the choice of $a, b, c$ in \eqref{2eq: a, b, c} such that neither $\varnu  $ nor $  \rho - \bfrac 1 2 $ is an integer, we denote
\begin{align}\label{2eq: defn of G(1), G(2)}
G_{\rho, \shskip \varnu }^{(1)} (z) = G_1 (a, b; c; z), \hskip 10 pt G_{\rho, \shskip \varnu }^{(2)} (z) = G_2 (a, b; c; z).
\end{align}
Let $ F_{\rho, \shskip \varnu }^{(1)} (z) $ and $ F_{\rho, \shskip \varnu }^{(2)} (z) $ be defined as in \eqref{1eq: defn of F(1,2)}. Note that their hypergeometric series expansions are absolutely convergent for all  $|z| \leqslant 1$ if $\Re \rho < \bfrac 1 2$. Further, we have
\begin{align}\label{2eq: G(1,2) and F(1,2)}
G_{\rho, \shskip \varnu }^{(1)} (z) = z^{-\frac 1 2 (\rho + \varnu )} F_{\rho, \shskip \varnu }^{(1)} (1/z), \hskip 10 pt G_{\rho, \shskip \varnu }^{(2)} (z) = z^{-\frac 1 2 (\rho - \varnu )} F_{\rho, \shskip \varnu }^{(2)} (1/z).
\end{align}

\begin{lem}\label{lem: hypergeometric equation} 
	Let $d$ be an integer. Let $\rho$, $\mu$ be such that $  \rho - \bfrac 1 2$ is not an integer and that $|\Re \mu | < \bfrac 1 2$ and $\mu \neq 0$.	Let $f (z)$ be a continuous function on the complex plane which is a solution of the following two differential equations,
	\begin{equation}\label{2eq: differential equations}
	\nabla_{ \rho, \shskip \mu + d} \, \tw = 0, \hskip 10 pt \overline \nabla_{ \rho, \shskip \mu - d} \, \tw = 0,
	\end{equation} 
	with the differential operators $\nabla_{ \rho, \shskip \mu + d} $ and $\overline \nabla_{ \rho, \shskip \mu - d}$ defined as in \eqref{2eq: nabla, 1} and \eqref{2eq: nabla, 2}. % Suppose that $f (z)$ is independent on $\arg z$ modulo $2 \pi$ {\rm(}in other words, $f (z)$ is a single-valued function{\rm)} when $|z| > 1$. 
	Suppose further that $f (z)$ admits the following asymptotic,
	\begin{align}\label{2eq: asymptoic f (z)}
	f ( z ) \sim c_1  z^{- \frac 1 2 ( \rho + \mu + d)} \shskip \widebar  z^{ \shskip - \frac 1 2 (  \rho + \mu - d )}  + c_2   z^{\shskip - \frac 1 2 ( \rho - \mu - d ) } \shskip \widebar z^{\shskip - \frac 1 2 (  \rho - \mu + d) }, \hskip 10 pt |z| \ra \infty.
	\end{align}
	Then  \begin{align}\label{2eq: identity f (z)}
	f ( z) = c_1 G_{\rho, \shskip \mu+d}^{(1)} (z)  G_{\rho, \shskip \mu-d}^{(1)} (\widebar z) + c_2  G_{\rho, \shskip \mu+d}^{(2)} (z)  G_{\rho, \shskip \mu-d}^{(2)} (\widebar z) 
	\end{align}
	for all $z$ in the complex plane.
\end{lem}

\begin{proof}
	By the theory of differential equations, it follows from \eqref{2eq: differential equations} that $f (z)$ may be uniquely written as a linear combination of $G^{(k)}_{\rho, \shskip  \mu + d} (z) G^{(l)}_{\rho, \shskip \mu - d} (\widebar z)$, with $k, l = 1, 2$, namely, 
	$$ f(z) = \underset{k,\shskip  l = 1,\shskip  2}{\sum \sum} \ c_{k l}\shskip  G^{(k)}_{\rho, \shskip  \mu + d} (z) G^{(l)}_{\rho, \shskip \mu - d} (\widebar z),$$
	for all $z$ away from the branch cut from $0$ to $1$. 
	
	In view of \eqref{2eq: G1 and G2} and \eqref{2eq: defn of G(1), G(2)}, it is clear that  $G_{\rho, \shskip \mu+d}^{(1)} (z)  G_{\rho, \shskip \mu-d}^{(1)} (\widebar z) $ and $  G_{\rho, \shskip \mu+d}^{(2)} (z)  G_{\rho, \shskip \mu-d}^{(2)} (\widebar z)$  are both single-valued real analytic functions when $|z| > 1$, and so are the quotient of $ G^{(1)}_{\rho, \shskip  \mu + d} (z) G^{(2)}_{\rho, \shskip \mu - d} (\widebar z)$ and $ (z/|z|)^{  -  \mu}$ and that of $ G^{(2)}_{\rho, \shskip  \mu + d} (z) G^{(1)}_{\rho, \shskip \mu - d} (\widebar z)$ and $ (z/|z|)^{   \mu}$. Since $   \mu $ is not an integer, the functions $ (z/|z|)^{  -  \mu}$ and  $ (z/|z|)^{   \mu}$ are not single-valued. Choosing $\arg z = 0, 2 \pi$, it follows from simple considerations that any nontrivial linear combination of $ G^{(1)}_{\rho, \shskip  \mu + d} (z) G^{(2)}_{\rho, \shskip \mu - d} (\widebar z) $ and $ G^{(2)}_{\rho, \shskip  \mu + d} (z) G^{(1)}_{\rho, \shskip \mu - d} (\widebar z) $ is not independent on $\arg z$ modulo $2 \pi$. So we must have $c_{12} = c_{21} = 0$. 
	
	Next, we need to prove that $c_{11} = c_1$ and $c_{22} = c_2$. For this, we deduce from  \eqref{2eq: G1 and G2} and \eqref{2eq: defn of G(1), G(2)} that
	\begin{align*}
	& G_{\rho, \shskip \mu+d}^{(1)} (z)  G_{\rho, \shskip \mu-d}^{(1)} (\widebar z) = |z|^{-  \rho -  \mu} (\widebar z/|z|)^{  d} + O \lp \big||z|^{-  \rho -   \mu - 1} \big| \rp, \\
	& G_{\rho, \shskip \mu+d}^{(2)} (z)  G_{\rho, \shskip \mu-d}^{(2)} (\widebar z) = |z|^{-  \rho +  \mu} ( z/|z|)^{ d} + O \lp \big||z|^{-   \rho +   \mu - 1} \big| \rp,
	\end{align*}
	as $|z| \ra \infty$. Since $|\Re \mu| < \bfrac 1 2$, the two error terms have order strictly lower than the least order of the two main terms. This forces $c_{11} = c_{1}$ and $c_{22} = c_{2}$ in order for $f (z)$ to have the prescribed asymptotic in \eqref{2eq: asymptoic f (z)}.
	
	It is now proven that the identity \eqref{2eq: identity f (z)} is valid except on the branch cut from $0$ to $1$. However $f (z)$ is continuous on the whole complex plane, the identity may be extended to the branch cut.
\end{proof}

We remark that the continuity of $f (z)$ is a very strong condition (it is not even clear now whether exists such a continuous solution $f (z)$ of the differential equations). It does not only force $c_{12} = c_{21} = 0$ in the proof above but also implies that the ratio $c_1 /c_2$ has to be unique as neither $  G_{\rho, \shskip \mu+d}^{(1)} (z)  G_{\rho, \shskip \mu-d}^{(1)} (\widebar z) $ nor $ G_{\rho, \shskip \mu+d}^{(2)} (z)  G_{\rho, \shskip \mu-d}^{(2)} (\widebar z)$ is single-valued in $|z| < 1$. Indeed, the following proposition demonstrates that the ratio must be equal to $	4^{-\mu}  C_{ \mu,\shskip  d}^{\shskip \rho}/ 4^{\mu}  C_{ - \mu,\shskip - d}^{\shskip \rho}$. This proposition however is not required in our proof of Theorem \ref{thm: general W-S formula}.

\begin{prop}\label{prop: hypergeometric |z| < 1}
	 Let $C_{ \mu,\shskip  d}^{\shskip \rho}$ be defined by {\rm\eqref{0eq: defn of C}}. Let $G_{\rho, \shskip \varnu }^{(1)} (z)$ and $G_{\rho, \shskip \varnu }^{(2)} (z)$ be given by {\rm\eqref{1eq: defn of F(1,2)}} and {\rm\eqref{2eq: G(1,2) and F(1,2)}}. Define
	\begin{equation}\label{1eq: defn E(1, 2)}
	E^{(1)}_{\rho, \shskip \varnu } (z) =  {_2F_1}\hskip -2 pt \lp \hskip - 0.5 pt \frac {\rho \hskip - 0.5 pt + \hskip - 0.5 pt\varnu } 2, \frac {\rho\hskip - 0.5 pt - \hskip - 0.5 pt\varnu  }2; \frac 1 2 ; z \hskip - 0.5 pt \rp  \hskip - 2 pt, \ E^{(2)}_{\rho, \shskip \varnu } (z) = {_2F_1} \hskip -2 pt \lp \hskip - 0.5 pt \frac {\rho\hskip - 0.5 pt + \hskip - 0.5 pt\varnu  \hskip - 0.5 pt + \hskip - 0.5 pt 1} 2, \frac {\rho\hskip - 0.5 pt - \hskip - 0.5 pt\varnu \hskip - 0.5 pt + \hskip - 0.5 pt 1}2; \frac 3 2 ; z \hskip - 0.5 pt \rp \hskip - 2 pt .
	\end{equation}
	\footnote{Note that, according to \cite[\S 2.2]{MO-Formulas}, $E^{(1)}_{\rho, \shskip \varnu } (z)$ and $\sqrt z E^{(2)}_{\rho, \shskip \varnu } (z)$ form a system of linearly independent solutions of the hypergeometric equation with $a, b, c$ as in \eqref{2eq: a, b, c} in the vicinity of the singular point $z = 0$.} We have
	\begin{equation}\label{3eq: hypergeometric, |z|<1}
	\begin{split}
& \quad \	4^{-\mu}  C_{ \mu,\shskip  d}^{\shskip \rho}  G_{\rho, \shskip \mu+d}^{(1)} (z)  G_{\rho, \shskip \mu-d}^{(1)}  (\widebar z)  +  4^{\mu}  C_{ - \mu,\shskip - d}^{\shskip \rho} G_{\rho, \shskip \mu+d}^{(2)} (z)  G_{\rho, \shskip \mu-d}^{(2)} (\widebar z)   \\
	& = B_{ \rho, \shskip \mu,\shskip  d} A_{ \rho, \shskip \mu,\shskip  d}^{(1)}  E_{\rho, \shskip \mu+d}^{(1)} (z)  E_{\rho, \shskip \mu-d}^{(1)} (\widebar z) + B_{ \rho, \shskip \mu,\shskip  d+1} A_{ \rho, \shskip \mu,\shskip  d}^{(2)} \cdot 4 |z| E_{\rho, \shskip \mu+d}^{(2)} (z)  E_{\rho, \shskip \mu-d}^{(2)} (\widebar z),
	\end{split}
	\end{equation}
	for $\arg z \in (0, 2 \pi)$,
	with 
	\begin{align}
	B_{ \rho, \shskip \mu,\shskip  d} = \frac {   (-)^{d}  \cos (\pi \mu) - \cos (\pi   \rho) } {4 \pi^{ 2 \rho+2}  } ,
	\end{align}
	\begin{equation}
	\begin{split}
	A_{ \rho, \shskip \mu,\shskip  d}^{(1)} =  \Gamma \lp \frac { \rho +   \mu+d } 2 \rp \Gamma \lp \frac { \rho +   \mu-d } 2 \rp \Gamma \lp \frac { \rho - \mu-d } 2 \rp \Gamma \lp \frac { \rho -  \mu+d } 2 \rp \hskip -1 pt,
	\end{split}
	\end{equation}
	and similarly
	\begin{equation}
	\begin{split}
	A_{ \rho, \shskip \mu,\shskip  d}^{(2)} = \Gamma \hskip -1 pt \lp \frac {1 \hskip -0.5 pt + \hskip -0.5 pt \rho \hskip -0.5 pt +  \hskip -0.5 pt \mu \hskip -0.5 pt + \hskip -0.5 pt d } 2 \rp \Gamma \hskip -1 pt \lp \frac {1 \hskip -0.5 pt + \hskip -0.5 pt \rho \hskip -0.5 pt +  \hskip -0.5 pt \mu \hskip -0.5 pt - \hskip -0.5 pt d } 2 \rp \Gamma \hskip -1 pt \lp \frac {1 \hskip -0.5 pt + \hskip -0.5 pt \rho \hskip -0.5 pt -  \hskip -0.5 pt \mu \hskip -0.5 pt - \hskip -0.5 pt d } 2 \rp \Gamma \hskip -1 pt \lp \frac {1 \hskip -0.5 pt + \hskip -0.5 pt \rho \hskip -0.5 pt -  \hskip -0.5 pt \mu \hskip -0.5 pt + \hskip -0.5 pt d } 2 \rp \hskip -1 pt .
	\end{split}
	\end{equation}
	Moreover, for $\Re \rho < \bfrac 1 2$, the combinations on both sides of \eqref{3eq: hypergeometric, |z|<1} give rise to a continuous function of $z$ on the whole complex plane, real-analytic except at $z = 1$.
\end{prop} 

\begin{proof}
%	Let $E^{(1)}_{\rho, \shskip \varnu } (z)$ and $E^{(2)}_{\rho, \shskip \varnu } (z)$ be defined as in \eqref{1eq: defn E(1, 2)}. 
By the transformation law \eqref{2eq: transformation z - 1/z}, we have
	\begin{align*}
	G^{(1)}_{\rho, \shskip \varnu } (z)  = \frac {\sqrt \pi \shskip \Gamma (1+\varnu )   e^{-\frac 1 2 \pi i (\rho+\varnu )} E^{(1)}_{\rho, \shskip \varnu } (z) } {\Gamma ((1+\rho+\varnu )/ 2) \Gamma ((2-\rho+\varnu ) / 2)}  + \frac {2 \sqrt \pi \shskip \Gamma (1+\varnu )   e^{-\frac 1 2 \pi i (\rho+\varnu +1)} \sqrt z E^{(2)}_{\rho, \shskip \varnu } (z) } {\Gamma ((\rho+\varnu )/ 2) \Gamma ((1-\rho+\varnu ) / 2)  } , \\
	G^{(2)}_{\rho, \shskip \varnu } (z)  = \frac {\sqrt \pi \shskip \Gamma (1-\varnu )   e^{-\frac 1 2 \pi i (\rho-\varnu )} E^{(1)}_{\rho, \shskip \varnu } (z) } {\Gamma ((1+\rho-\varnu )/ 2) \Gamma ((2-\rho-\varnu ) / 2)}  + \frac {2 \sqrt \pi \shskip \Gamma (1-\varnu )   e^{-\frac 1 2 \pi i (\rho-\varnu +1)} \sqrt z E^{(2)}_{\rho, \shskip \varnu } (z) } {\Gamma ((\rho-\varnu )/ 2) \Gamma ((1-\rho-\varnu ) / 2) } ,
	\end{align*}
	for $\arg z \in (0, 2 \pi)$. For $ G^{(1)}_{\rho, \shskip \varnu } (\widebar z) $  and $G^{(1)}_{\rho, \shskip \varnu } (\widebar z)$, the formulae are similar but in a subtle way---the factors like $  e^{-\frac 1 2 \pi i (\rho + \varnu )} $ should be replaced by  $  e^{\frac 1 2 \pi i (\rho+\varnu )} $. This is because,  in view of the conditions in  \eqref{2eq: transformation z - 1/z}, we need to let  $|\arg ( - 1/\widebar z)| < \pi $ and hence $-1 = e^{-i \pi}$. 
	
	Applying  these to the left hand side of \eqref{3eq: hypergeometric, |z|<1}.
	By  the duplication formula for the gamma function, we find that,  up to sign,  the coefficients of $  E^{(1)}_{\rho, \shskip \mu+d} (z) \cdot \sqrt {\widebar z} E^{(2)}_{\rho, \shskip \mu-d} (\widebar z) $ and $  \sqrt {  z} E^{(2)}_{\rho, \shskip \mu+d} (z) \cdot  E^{(1)}_{\rho, \shskip \mu-d} (\widebar z) $ are both equal to  
	\begin{align*}
	\frac {  i^{\shskip 2d-1} } {2 \pi ^{ 2 \rho} \sin (\pi \mu)}    \bigg( &
	\frac {   \sin (\pi ( \rho  +  \mu)) \Gamma ((\rho +   \mu + d)/2 )  \Gamma ((1 +\rho +   \mu - d)/2 )  } 
	{    \Gamma ((2-\rho+\mu+d) / 2) \Gamma ((1-\rho+\mu-d) / 2) }  \\
	& - \frac { \sin (\pi ( \rho  - \mu)) \Gamma ((\rho - \mu - d)/2 )  \Gamma ((1 +\rho -   \mu + d)/2 )  } 
{    \Gamma ((2-\rho-\mu-d) / 2) \Gamma ((1-\rho-\mu+d) / 2) } \bigg),
	\end{align*}
which in turn is equal to zero by Euler's reflection formula for the gamma function. We obtain \eqref{3eq: hypergeometric, |z|<1} by calculating the coefficients of $  E^{(1)}_{\rho, \shskip \varnu } (z)       E^{(1)}_{\rho, \shskip \varnu } (\widebar z) $ and $    |z| E^{(2)}_{\rho, \shskip \varnu } (z)   E^{(2)}_{\rho, \shskip \varnu } (\widebar z) $ in a similar way. % Note that we assumed that $\arg z \in (0, 2 \pi)$ when \eqref{2eq: transformation z - 1/z} was applied. Nevertheless, 

It was noted before that the left hand side of \eqref{3eq: hypergeometric, |z|<1} gives rise to a single-valued continuous function on $|z| \geqslant 1$, real-analytic in the interior $|z| > 1$, provided that $\Re \rho < \bfrac 1 2$.  It is also clear that the same is true for the right hand side of \eqref{3eq: hypergeometric, |z|<1} on  $|z| \leqslant 1$. Thus \eqref{3eq: hypergeometric, |z|<1} is valid and defines a continuous function (real-analytic except at $z = 1$) on the whole complex plane if $\Re \rho < \bfrac 1 2$. 
\end{proof}

	\section{Proof of Theorem \ref{thm: general W-S formula}} \label{sec: Proof of Theorem 1}

	Recall that $\bfJ_{ \mu,\shskip  d} (z) $ and $ \boldsymbol {R}_{ \mu,\shskip  d} (z) $ were defined as in (\ref{0def: J mu m (z)}, \ref{0eq: defn of Bessel}) and (\ref{0eq: defn P mu m (z)},\ref{0eq: defn of R}) respectively. Let
		\begin{equation}\label{3eq: defn of F}
	\bfF_{\mu,\shskip  d }^{\shskip \rho} \big( y e^{i \theta} \big) =  \int_{0}^{2 \pi} \int_0^\infty \bfJ_{ \mu,\shskip  d }    \big(  x e^{i\phi} \big)  e (- 2 x y \cos (\phi + \theta) )  x^{2 \rho - 1}  d x \shskip d \phi,
	\end{equation}
	and
	\begin{equation}\label{3eq: defn of G}
	\boldsymbol {P}_{ \mu,\shskip  d}^{\shskip \rho} \big(y e^{i \theta}\big) =  \int_{0}^{2 \pi} \int_0^\infty \boldsymbol {R}_{ \mu,\shskip  d}   \big(  x e^{i\phi} \big)  e (- 2 x y \cos (\phi + \theta) )  x^{2 \rho - 1}  d x \shskip d \phi . 
	\end{equation} 
Since $ \bfJ_{ \mu,\shskip  d } (z) $ is an even function, the integral $ \bfF_{\mu,\shskip  d }^{\shskip \rho} \big(y e^{i \theta}\big) $ is also even in the sense that it is independent on $\theta$ modulo $\pi$.

	First of all, the convergence of  $\bfF_{\mu,\shskip  d }^{\shskip \rho} \big( y e^{i \theta} \big)$ and   $\bfP_{\mu,\shskip  d }^{\shskip \rho} \big( y e^{i \theta} \big)$ may be easily examined with the help of Lemma \ref{lem: integration by parts}, along with \eqref{2eq: bound for J mu m, weak, 2} and \eqref{1eq: J = W + W}-\eqref{1eq: derivatives of E}  for the former. %We leave the proof of the following lemma to the reader.
	
	\begin{lem}\label{lem: convergence of F and G}
		Let the integrals $\bfF_{\mu,\shskip  d }^{\shskip \rho} \big( y e^{i \theta} \big)$ and   $\bfP_{\mu,\shskip  d }^{\shskip \rho} \big( y e^{i \theta} \big)$ be defined as above. 
		
		{\rm(1).}  $\bfF_{\mu,\shskip  d }^{\shskip \rho} \big( y e^{i \theta} \big)$ is absolutely convergent if $|\Re \mu | <   \Re \rho < \bfrac 1 2$, and, for $y > 2$, convergent if $|\Re \mu | <   \Re \rho < 1$,  uniformly for $\Re \mu$ and $\Re \rho$ in compact sets.
		
		{\rm(2).} $\bfP_{\mu,\shskip  d }^{\shskip \rho} \big( y e^{i \theta} \big)$ is convergent for $y > 0$ if $|\Re \mu | <   \Re \rho < \bfrac 1 2 - |\Re \mu |$, uniformly for $\Re \mu$ and $\Re \rho$ in compact sets.
		
\noindent		Moreover, both $\bfF_{\mu,\shskip  d }^{\shskip \rho} \big( y e^{i \theta} \big)$ and  $\bfP_{\mu,\shskip  d }^{\shskip \rho} \big( y e^{i \theta} \big)$ are continuous functions when $y$ is in the indicated ranges above.
	\end{lem}

\begin{proof}
	We partition the integrals  $\bfF_{\mu,\shskip  d }^{\shskip \rho} \big( y e^{i \theta} \big)$ and   $\bfP_{\mu,\shskip  d }^{\shskip \rho} \big( y e^{i \theta} \big)$ according to a partition of unity $(1- \tw(x)) + \tw(x) \equiv 1$  for $ \lp 0 , \infty \rp =  (  0, 2 ] \cup [1, \infty)$,  say. 
	
	By \eqref{2eq: bound for J mu m, weak, 2}, it is obvious that the integral
	\begin{align*}
	\int_{0}^{2 \pi} \int_0^2 (1- \tw(x)) \bfJ_{ \mu,\shskip  d }    \big(  x e^{i\phi} \big)  e (- 2 x y \cos (\phi + \theta) )  x^{2 \rho - 1}  d x \shskip d \phi
	\end{align*}
	 is convergent if $ |\Re \mu | < \Re \rho $; the same statement is true if $ \bfJ_{ \mu,\shskip  d }  $ is substituted by $ \bfR_{ \mu,\shskip  d } $.
	 
	 Next, we consider the following integral
	 \begin{align*}
	 \int_{0}^{2 \pi} \int_{1}^{\infty}   \tw (x)  \bfJ_{ \mu,\shskip  d } \big(  x e^{i\phi} \big)       e (- 2 x y \cos (\phi + \theta) )  x^{2 \rho - 1}  d x \shskip d \phi. 
	 \end{align*}
Since $ \bfJ_{ \mu,\shskip  d } (z) = O (1/|z|) $ for $|z| \geqslant 1$, this integral is absolutely convergent provided that $\Re \rho < \bfrac 1 2$. According to \eqref{1eq: J = W + W}, we divide it into three similar integrals which contain $ \boldsymbol{W} \big( x e^{i\phi} \big)$, $\boldsymbol{W} \left( - x e^{i\phi} \right)$ and $\bfE_1 \big(  x e^{i\phi} \big) $ respectively. For the first two integrals, we apply Lemma \ref{lem: integration by parts} with $a = - 2 y$, $b= \pm 4$, $f (z) = \tw (|z|) \boldsymbol{W} (\pm z)  |z|^{2 \rho-2}$ and $ \gamma =  \Re \rho - \bfrac 1 2 $ (by \eqref{1eq: derivatives of W}), and it follows that these two integrals are convergent if $\Re \rho < 1$ ($ \gamma < \bfrac 1 2$). By \eqref{1eq: derivatives of E}, the third integral is absolutely convergent when  $\Re \rho < 1$. %Moreover, the first two integrals are  absolutely convergent when  $\Re \rho < \bfrac 1 2$.
	 
	 Similarly, applying  Lemma \ref{lem: integration by parts} with $a = - 2 y$, $b= 0 $, $f (z) = \tw (|z|) \bfR_{ \mu,\shskip  d } (z)  |z|^{2 \rho-2}$ and $ \gamma =  \Re \rho + |\Re \mu | $, we infer that the integral 
	 \begin{align*}
	 \int_{0}^{2 \pi} \int_{1}^{\infty}   \tw (x)  \bfR_{ \mu,\shskip  d } \big(  x e^{i\phi} \big)       e (- 2 x y \cos (\phi + \theta) )  x^{2 \rho - 1}  d x \shskip d \phi
	 \end{align*}
	 is convergent provided that $ \Re \rho <  \bfrac 1 2 - |\Re \mu |$.
\end{proof}

	\subsection{Asymptotic of $\bfF_{\mu,\shskip  d }^{\shskip \rho} \big(y e^{i \theta}\big)$}\label{sec: asymptotic}
	
	In this section, we assume that  $ \lambdaup =  |\Re \mu | < \bfrac 1 4$, $\beta = \Re \rho$ and  $ \lambdaup < \beta < \bfrac 1 2 - \lambdaup$ so that   $\bfF_{\mu,\shskip  d }^{\shskip \rho} \big( y e^{i \theta} \big)$ and  $\bfP_{\mu,\shskip  d }^{\shskip \rho} \big( y e^{i \theta} \big)$ are both convergent. Moreover, let $\mu \neq 0$ for simplicity.
	
	Our first aim is to prove the following asymptotic of $\bfF_{\mu,\shskip  d }^{\shskip \rho}( y e^{i \theta} )$,
	\begin{equation}\label{3eq: asymptotic F (y)}
	\bfF_{\mu,\shskip  d }^{\shskip \rho} \big( y e^{i \theta} \big) \sim  C_{\mu, \shskip d}^{\shskip \rho}   y^{-  2 \mu - 2\rho} e^{- 2 i d \theta}  + C_{- \mu, \shskip - d}^{\shskip \rho}       { y^{\shskip  2  \mu - 2\rho} e^{2i d \theta}} , \hskip 10 pt y \ra \infty,
	\end{equation}
	in which $ C_{\mu, \shskip d}^{\shskip \rho} $ is defined by \eqref{0eq: defn of C} in Theorem \ref{thm: general W-S formula}. It is clear that \eqref{3eq: asymptotic F (y)} follows from the following two lemmas.
	\begin{lem}\label{lem: F sim G} %Let the notation and assumptions be as above. 
		We have
		\begin{align}\label{3eq: F sim G}
		\bfF_{\mu,\shskip  d }^{\shskip \rho} \big( y e^{i \theta} \big) = \bfP_{\mu,\shskip  d }^{\shskip \rho} \big( y e^{i \theta} \big) + o \big( y^{- 2 \lambdaup - 2 \beta} \big), \hskip 10 pt y \ra \infty.
		\end{align}
	\end{lem}

	\begin{lem}\label{lem: formula for G} %Suppose that $|\Re \mu | <   \Re \rho < \bfrac 1 2 - |\Re \mu |$. 
		We have
		\begin{align}\label{3eq: G = y}
		\bfP_{ \mu,\shskip  d}^{\shskip \rho} \big(y e^{i \theta}\big)  =   C_{\mu, \shskip d}^{\shskip \rho}   y^{- 2 \mu - 2\rho} e^{- 2 i d \theta}  + C_{- \mu, \shskip - d}^{\shskip \rho}       { y^{\shskip  2 \mu - 2\rho} e^{2i d \theta}}  .
		\end{align} 
	\end{lem} 
	
%	We shall prove the asymptotic formula \eqref{3eq: F sim G} in \S \ref{sec: asymptotic F sim G} and the identity \eqref{3eq: G = y} in \S \ref{sec: formua for G}. 
	
%	\vskip 5 pt

%	\subsubsection{}\label{sec: asymptotic F sim G} 
\begin{proof}[Proof of Lemma \ref{lem: F sim G}]
Let  $y > 2$ be sufficiently large.  All the implied constants in our computation will only depend on  $\beta$, $\lambdaup$ and $ d $. 
	
	We split $\bfF_{\mu,\shskip  d }^{\shskip \rho} \big(y e^{i \theta}\big)   - \bfP_{\mu,\shskip  d }^{\shskip \rho} \big(y e^{i \theta}\big)  $ as the sum
	\begin{align*}
	\bfF_{\mu,\shskip  d }^{\shskip \rho}  \big(y e^{i \theta}\big)   - \ & \bfP_{\mu,\shskip  d }^{\shskip \rho} \big(y e^{i \theta}\big)   =  \bfD_{\mu,\shskip  d }^{\shskip \rho}  \big(y e^{i \theta}\big)  + \acute{\bfE}_{\mu,\shskip  d }^{\shskip \rho}  \big(y e^{i \theta}\big)  + \grave{\bfE}_{\mu,\shskip  d }^{\shskip \rho}  \big(y e^{i \theta}\big)  - \hat{\bfE}_{\mu,\shskip  d }^{\shskip \rho} \big(y e^{i \theta}\big)  \\
	 =\ & \int_{0}^{2 \pi} \int_0^{2y^{  - \frac 1 2}}  u (x) \big( \bfJ_{ \mu,\shskip  d } \big(  x e^{i\phi} \big) -  \bfR_{ \mu,\shskip  d } \big(  x e^{i\phi} \big) \big)      e (- 2 x y \cos (\phi + \theta) )  x^{2 \rho - 1}  d x \shskip d \phi\\
	& + \int_{0}^{2 \pi} \int_{y^{ - \frac 1 2}}^{2} \tv (x)  \bfJ_{ \mu,\shskip  d } \big(  x e^{i\phi} \big)       e (- 2 x y \cos (\phi + \theta) )  x^{2 \rho - 1}  d x \shskip d \phi\\
	& +  \int_{0}^{2 \pi} \int_{1}^{\infty}   \tw (x)  \bfJ_{ \mu,\shskip  d } \big(  x e^{i\phi} \big)       e (- 2 x y \cos (\phi + \theta) )  x^{2 \rho - 1}  d x \shskip d \phi\\
	& -  \int_{0}^{2 \pi} \int_{y^{- \frac 1 2}}^{\infty}   (\tv (x) + \tw (x) ) \bfR_{ \mu,\shskip  d } \big(  x e^{i\phi} \big)       e (- 2 x y \cos (\phi + \theta) )  x^{2 \rho - 1}  d x \shskip d \phi, 
	\end{align*}
	where  $ u (x) + \tv (x) + \tw (x) \equiv 1 $
	is a partition of unity on $ \lp 0 , \infty \rp$ such that $u (x) $, $v (x)  $ and $w (x) $ are smooth functions supported on $\big(  0, 2 y^{ - \frac 1 2 } \big]$,       $ \big[  y^{  - \frac 1 2 } , 2 \big] $ and $  \left[  1, \infty \rp$, respectively, and that $x^{r} u^{(r)} (x), x^r \tv^{(r)}  (x), x^r \tw^{(r)}   (x) \Lt_{\, r} 1 $. 
	
	First, it follows from \eqref{2eq: bound for J mu m, weak} that
	\begin{align*}
	\bfD_{\mu,\shskip  d }^{\shskip \rho}  \big(y e^{i \theta}\big) \Lt \int_{0}^{2 \pi} \int_0^{2y^{  - \frac 1 2}}     x^{2 \beta - 2 \lambdaup + 1}  d x \shskip d \phi \Lt y^{  \lambdaup - \beta  - 1} = y^{- 2 \lambdaup - 2 \beta} y^{\shskip   3 \lambdaup + \beta - 1} .
	\end{align*}
	Since $\beta  < \bfrac 1 2 -  \lambdaup$ and  $\lambdaup < \bfrac 1 4$, we have $3 \lambdaup + \beta - 1 < 2\lambdaup - \bfrac 1 2  < 0$ and hence $\bfD_{\mu,\shskip  d }^{\shskip \rho}  \big(y e^{i \theta}\big) = o \big( y^{- 2 \lambdaup - 2 \beta} \big)$ as desired.
	
	Second, applying Lemma \ref{lem: integration by parts} with $a = - 2 y$, $b=0$, $A = y^{- \frac 1 2}$, $B = 2$, $M = 1$, $f (z) = \tv (|z|)  \bfJ_{ \mu,\shskip  d }  (z) |z|^{2 \rho-2}$ and $ \gamma =  \beta -  \lambdaup $ (see \eqref{2eq: bound for J mu m, weak, 2}), we deduce that
	\begin{align*}
	\acute{\bfE}_{\mu,\shskip  d }^{\shskip \rho}  \big(y e^{i \theta}\big) \Lt y^{ (1 - \beta + \lambdaup) - 2 } = y^{ \lambdaup - \beta  - 1 }.
	\end{align*}
	As above, we also have $\acute{\bfE}_{\mu,\shskip  d }^{\shskip \rho}  \big(y e^{i \theta}\big) = o \big( y^{- 2 \lambdaup - 2 \beta} \big)$. Likewise,   $\hat{\bfE}_{\mu,\shskip  d }^{\shskip \rho}  \big(y e^{i \theta}\big) = o \big( y^{- 2 \lambdaup - 2 \beta} \big)$. 
	
	Third, according to \eqref{1eq: J = W + W}, we divide $ \grave{\bfE}_{\mu,\shskip  d }^{\shskip \rho}  \big(y e^{i \theta}\big) $ into three similar integrals which contain $ \boldsymbol{W} \big( x e^{i\phi} \big)$, $\boldsymbol{W} \left( - x e^{i\phi} \right)$ and $\bfE_1 \big(  x e^{i\phi} \big) $ respectively. For the first two integrals, we apply Lemma \ref{lem: integration by parts} with $a = - 2 y$, $b= \pm 4$, $A = 1$, $B = \infty$, $M = 1$, $f (z) = \tw (|z|) \boldsymbol{W} (\pm z)  |z|^{2 \rho-2}$ and $ \gamma =  \beta - \bfrac 1 2 $ (by \eqref{1eq: derivatives of W}), and it follows that these two integrals are both $O (y^{-2})$. For the third integral containing $\bfE_1 \big(  x e^{i\phi} \big) $,  we apply Lemma \ref{lem: integration by parts} with $a = - 2 y$, $b= 0 $, $A = 1$, $B = \infty$, $M = 1$, $f (z) = \tw (|z|) \bfE_1 ( z) |z|^{2 \rho-2}$ and $ \gamma =  \beta  $ (by \eqref{1eq: derivatives of E}), and it follows that the third integral is also $O (y^{-2})$. We conclude that	
	 $$\grave{\bfE}_{\mu,\shskip  d }^{\shskip \rho}  \big(y e^{i \theta}\big) = O (y^{-2})= o \big( y^{-2 \lambdaup - 2 \beta} \big)$$ as $\beta < \bfrac 12 - \lambdaup$. %By choosing $N = 2M - 1$, we infer that $ \grave{\bfE}_{\mu,\shskip  d }^{\shskip \rho} \big(y e^{i \theta}\big) $ is arbitrarily small in the sense that $ \grave{\bfE}_{\mu,\shskip  d }^{\shskip \rho} \big(y e^{i \theta}\big) = O_{M} (y^{- M}) $ for arbitrary positive integer $M$.
	 
	 Finally, combining the foregoing results, we obtain the asymptotic formula \eqref{3eq: F sim G} in Lemma \ref{lem: F sim G}.\footnote{Even more, it may be proven that
	 that $\acute{\bfE}_{\mu,\shskip  d }^{\shskip \rho}  \big(y e^{i \theta}\big)$, $ \grave{\bfE}_{\mu,\shskip  d }^{\shskip \rho}  \big(y e^{i \theta}\big) $ and $ \hat{\bfE}_{\mu,\shskip  d }^{\shskip \rho} \big(y e^{i \theta}\big)$ are all arbitrarily small, namely, $O_{A} (y^{- A})$ for arbitrary $A > 0$.}  %$\acute{\bfE}_{\mu,\shskip  d }^{\shskip \rho}  \big(y e^{i \theta}\big) $, $ \grave{\bfE}_{\mu,\shskip  d }^{\shskip \rho}  \big(y e^{i \theta}\big) $ and $ \hat{\bfE}_{\mu,\shskip  d }^{\shskip \rho} \big(y e^{i \theta}\big) $ are indeed arbitrarily small. 
	 \end{proof}

%	\subsubsection{}\label{sec: formua for G}
	
	The formula \eqref{3eq: G = y} in Lemma \ref{lem: formula for G} is an immediate consequence of the following lemma, applied with $m = \pm 2 d$  and $\varnu   =  \rho \pm \mu$.

	\begin{lem}\label{lem: integral of P}
		Let $m$ be an integer. Let $- \bfrac 1 2 |m| <   \Re \varnu   < \bfrac 3 4$. We have
		\begin{equation*}
		 \begin{split}
		 \int_0^\infty \int_{0}^{2 \pi}   e (- 2 x y \cos (\phi + \theta) )  x^{2 \varnu   - 1} e^{i m \phi} d \phi d x = \left\{ 
		 \begin{split}
		 &\frac {\sin (\pi \varnu ) \Gamma (\varnu  + m/2) \Gamma ( \varnu  - m/2) } {(2 \pi )^{2 \varnu }   y^{2 \varnu } e^{im \theta} },  \\
		 &\frac { i \cos (\pi \varnu ) \shskip \Gamma (\varnu  + m/2) \Gamma ( \varnu  - m/2) } {  (2 \pi )^{2 \varnu }  y^{2 \varnu } e^{im \theta} },
		 \end{split}
		 \right.
		 \end{split}
		\end{equation*}
		according as $m$ is even or odd{\rm;}
		the integral on the left is convergent as iterated integral for $- \bfrac 1 2 |m| <   \Re \varnu   < \bfrac 3 4$ and as double integral only for  $ 0 <   \Re \varnu   < \bfrac 1 2$.
		
	\end{lem}
	
	\begin{proof}
		The inner integral over $\phi$ may be evaluated by \eqref{2eq: integral repn of J} so that
		\begin{align*}
		\int_0^\infty \int_{0}^{2 \pi}   e (- 2 x y \cos (\phi + \theta) )  e^{ i m \phi} d \phi \, x^{2 \varnu   - 1} d x  = 2 \pi (-i)^{|m|} e^{- im \theta}  \int_0^\infty J_{|m|} (4 \pi x y) x^{2 \varnu  - 1} d x.
		\end{align*}
		In view of \eqref{1eq: WSS formula}, the integral in the right hand side is  equal to 
		\begin{equation*}
\frac { \Gamma (\varnu  + |m|/2) } {2(2 \pi y)^{2 \varnu } \Gamma (1 - \varnu  + |m|/2)} = 
\left\{ 
\begin{split}
 &\frac { \sin (\pi \varnu ) \Gamma (\varnu  + |m|/2) \Gamma ( \varnu  - |m|/2) } {2 \pi (2 \pi y)^{2 \varnu } i^{|m|}  }, \hskip 10 pt \text{ if } m \text{ is even,} \\
 &\frac { \cos (\pi \varnu ) \Gamma (\varnu  + |m|/2) \Gamma ( \varnu  - |m|/2) } {2 \pi (2 \pi y)^{2 \varnu } i^{|m| + 1}  }, \hskip 8.5 pt \text{ if } m \text{ is odd,}
\end{split}
\right.
		\end{equation*}
		as desired.
	\end{proof}

\subsection{Differential equations for $\bfF_{\mu,\shskip  d }^{\shskip \rho}  $}\label{sec: differential equations} Since $ \bfF_{\mu,\shskip  d }^{\shskip \rho} \big(y e^{i \theta}\big) $ is even, it would be convenient to set $4 u = y^2 e^{ 2 i\theta} $ and consider $\bfF_{\mu,\shskip  d }^{\shskip \rho} (2 \sqrt u  )  $.  Assume that $|\Re \mu | <   \Re \rho < \bfrac 1 2$ so that the integral $\bfF_{\mu,\shskip  d }^{\shskip \rho} (2 \sqrt u  )  $ is always convergent and gives rise to a continuous function.
We are now going to verify
\begin{align}\label{5eq: differential equations}
\nabla_{\rho, \shskip \mu + d} \big( \bfF_{\mu,\shskip  d }^{\shskip \rho}  (2 \sqrt u  )  \big) = 0, \hskip 10 pt \overline \nabla_{\rho, \shskip  \mu - d}  \big( \bfF_{\mu,\shskip  d }^{\shskip \rho}  (2 \sqrt u  )  \big) = 0,
\end{align}
with the hypergeometric differential operators $\nabla_{\rho, \shskip \mu + d} $ and $ \overline \nabla_{\rho, \shskip  \mu - d}  $ defined as in \eqref{2eq: nabla, 1} and \eqref{2eq: nabla, 2} respectively.
By symmetry, we only need to verify the former, which, if we set $\varnu  = \mu + d$ for simplify, may be explicitly written as
\begin{align*}%\label{5eq: differential equations, explicit}
4 u (1-u) \frac {\partial^2 \bfF_{\mu,\shskip  d }^{\shskip \rho}  (2 \sqrt u  ) } {\partial u^2}   + 2( 1 - 2 (\rho + 1) u ) \frac {\partial \bfF_{\mu,\shskip  d }^{\shskip \rho}  (2 \sqrt u  ) } {\partial u}  -  \lp \rho^2 - \varnu ^2 \rp \bfF_{\mu,\shskip  d }^{\shskip \rho} (2 \sqrt u  ) = 0.
\end{align*}

	For $s  , r = 0, 1$, $2$, with $s + r = 0, 1, 2$, we introduce
\begin{align*}%\label{3eq: F s r nu (y)}
\bfI_{s ,\shskip  r,\shskip  \mu, \shskip  d}^{\shskip \rho} (u) =  \frac 1 {2 \sqrt{ u^{s } } }   \sideset{}{_{\BC \smallsetminus \{0 \}}}{\iint}  z^{\shskip\rho + s  +  r - 1} \widebar {z}^{\shskip \rho - 1} (\partial/\partial z)^{r} \bfJ_{ \mu, \shskip  d} (z) e \big( \hskip - 2 pt -    2 \big( z \sqrt{u} + {\widebar z} \sqrt{  \widebar u} \big)    \big) i d z \nwedge d \widebar z.
\end{align*}
The integral $\bfI_{s ,\shskip  r,\shskip  \mu, \shskip  d}^{\shskip \rho} (u) $ should be regarded as distribution on $\BC$. Precisely, for any Schwartz function $ f (u) \in  \SS (\BC  )$ (the Schwartz space on $\BC$ is denoted by $ \SS (\BC  ) $ as usual), let
\begin{align*}
\big\langle u^{t} \bfI_{s ,\shskip  r,\shskip  \mu, \shskip  d}^{\shskip \rho}  , f   \big\rangle =   \sideset{}{_{\BC \smallsetminus \{0 \}}}{\iint}  z^{\shskip\rho + s  +  r - 1} \widebar {z}^{\shskip \rho - 1} (\partial/\partial z)^{r}  \bfJ_{ \mu, \shskip  d} (z) f^{\scriptscriptstyle\sharp} (z)   { i d z \nwedge d \widebar z},
\end{align*}
with $ f^{\scriptscriptstyle\sharp} (z) \in  \SS (\BC  )$ given by
\begin{align*}
f^{\scriptscriptstyle\sharp} (z) = \sideset{}{_{ \BC \smallsetminus \{0 \} }}{\iint}   { u^{1+2t-s} \widebar u f (u^2)  e  (   -    2  ( z  {u} + {\widebar z}  {  \widebar u}  )     )}     {i d u \nwedge d \widebar u},
\end{align*}
in which $t = 0, 1, 2$ or  $t = -1$ only if $s = 0, 1$.
Note here that $ f^{\scriptscriptstyle\sharp} (z /2)$ is the Fourier transform of $  u^{1+2t-s} \widebar u f (u^2)  $ and hence is also Schwartz. In the theory of distributions, we are free to differentiate under the integral and integrate by parts in the formal manner as what we shall do in the following. %; in particular, the differential equation makes sense.

To start with, note that $ \bfF_{\mu,\shskip  d }^{\shskip \rho} (2 \sqrt u  )  = \bfI_{0 ,\shskip  0,\shskip  \mu, \shskip  d}^{\shskip \rho} (u) $.
For brevity, we put $\bfI_{s ,\shskip  r}  = \bfI_{s ,\shskip  r,\shskip  \mu, \shskip  d}^{\shskip \rho} $ and $\bfI_{s } = \bfI_{s , \shskip 0}$.

First, for $s  = 0, 1$, we have
\begin{align*}
\frac {\partial \bfI_{s  }} {\partial u} = - \frac {s} {2 u } \bfI_{s } - 2 \pi i \bfI_{s + 1 }
\end{align*}
and
\begin{align*}
\frac {\partial^2 \bfI_{0 }} {\partial u^2} & = - 2 \pi i \frac {\partial \bfI_{1 }} {\partial u} =  \frac { \pi i} {u } \bfI_{1 } - 4 \pi^2 \bfI_{2} .
\end{align*}
Second, for $s, r  = 0 , 1 $, with $s + r = 0, 1$,
by partial integration, 
\begin{align*}%\label{3eq: F s r nu (y)}
&   4 \pi i  \sqrt u  \sideset{}{_{\BC \smallsetminus \{0 \}}}{\iint}  z^{\shskip\rho + s  +  r} \shskip \widebar {z}^{\shskip \rho - 1} (\partial/\partial z)^{r} \bfJ_{ \mu, \shskip  d} (z)  e \big( \hskip - 2 pt -    2 \big( z \sqrt{u} + {\widebar z} \sqrt{  \widebar u} \big)    \big) i d z \nwedge d \widebar z\\
 = \ & (\rho + s  +  r) \sideset{}{_{\BC \smallsetminus \{0 \}}}{\iint}  z^{\shskip\rho + s  +  r -1} \widebar {z}^{\shskip \rho - 1} (\partial/\partial z)^{r} \bfJ_{ \mu, \shskip  d} (z) e \big( \hskip - 2 pt -    2 \big( z \sqrt{u} + {\widebar z} \sqrt{  \widebar u} \big)    \big) i d z \nwedge d \widebar z \\
& \hskip 33 pt + \sideset{}{_{\BC \smallsetminus \{0 \}}}{\iint}  z^{\shskip\rho + s  +  r} \shskip \widebar {z}^{\shskip \rho - 1} (\partial/\partial z)^{r+1} \bfJ_{ \mu, \shskip  d} (z) e \big( \hskip - 2 pt -    2 \big( z \sqrt{u} + {\widebar z} \sqrt{  \widebar u} \big)    \big) i d z \nwedge d \widebar z. 
\end{align*}
It follows that
\begin{align*}
  4 \pi i u   \bfI_{s +1,\shskip  r } =  ( \rho +  s  +  r) \bfI_{s ,\shskip  r } +   \bfI_{s ,\shskip  r+1 }, 
\end{align*}
and hence
\begin{align*}
-   {16 \pi^2} {u^2} \bfI_{2 }   =     {4 \pi i } (\rho+1) u \bfI_{1  } +  {4 \pi i} u \bfI_{1 ,\shskip   1 }   =  (\rho^2+\rho)  \bfI_0 + (2\rho+2) \bfI_{0, \shskip 1} +\bfI_{0, \shskip 2} .
\end{align*}
Third, since $\nabla_{  \mu + d} \lp \bfJ_{\mu,\shskip  d} \big(   z  / 4\pi  \big) \rp = 0$ (see \eqref{2eq: nabla} and \eqref{2eq: nabla J = 0}), we have
\begin{align*}
\bfI_{0, \shskip 2} + \bfI_{0, \shskip 1} + 16 \pi^2 u \bfI_{2, \shskip 0} - \varnu ^2 \bfI_{0, \shskip 0} = 0 \hskip 10 pt (\varnu  = \mu+ d).
\end{align*}
Finally, combining these, we have
\begin{align*}
& \hskip 13 pt 4 u (1-u) \frac {\partial^2 \bfI_0 } {\partial u^2}   + 2( 1 - 2 (\rho + 1) u ) \frac {\partial \bfI_0 } {\partial u}  -  \lp \rho^2 - \varnu ^2 \rp \bfI_0 \\
& = 4 u (1-u) \lp \frac { \pi i} {u } \bfI_{1 } - 4 \pi^2 \bfI_{2} \rp - 4 \pi i ( 1 - 2 (\rho + 1) u )  \bfI_{1 }  -  \lp \rho^2 - \varnu ^2 \rp \bfI_0\\ 
& = - 16  \pi^2 u   \bfI_{2} + 16  \pi^2 u^2   \bfI_{2} + 4 \pi i ( 2 \rho + 1 ) u \bfI_{1 }  -  \lp \rho^2 - \varnu ^2 \rp \bfI_0 \\
& = - 16  \pi^2 u   \bfI_{2} -  ( (\rho^2 \hskip -1 pt + \hskip -1 pt \rho) \bfI_0 \hskip -1 pt + \hskip -1 pt (2\rho \hskip -1 pt + \hskip -1 pt 2) \bfI_{0, \shskip 1} \hskip -1 pt + \hskip -1 pt\bfI_{0, \shskip 2} ) +   ( 2 \rho + 1 )  (\rho \bfI_0 \hskip -1 pt + \hskip -1 pt \bfI_{0, \shskip 1} )  -  \lp \rho^2 \hskip -1 pt - \hskip -1 pt \varnu ^2 \rp \bfI_0 \\
& = \varnu ^2 \bfI_0 - 16  \pi^2 u   \bfI_{2} -   \bfI_{0, \shskip 1} - \bfI_{0, \shskip 2}  \\
& = 0,
\end{align*} 
as desired.

\subsection{Conclusion} Combining \eqref{3eq: asymptotic F (y)} and \eqref{5eq: differential equations}, we deduce from Lemma \ref{lem: hypergeometric equation} that,  under the conditions $|\Re \mu | <   \Re \rho < \bfrac 1 2$ and $\mu \neq 0$,
\begin{equation}\label{3eq: F = C GG}
\bfF_{\mu,\shskip  d }^{\shskip \rho} (2\sqrt{u}) = 4^{-\mu-\rho} C_{\mu, \shskip d}^{\shskip \rho} \cdot  G_{\rho, \shskip \mu+d}^{(1)} (u)  G_{\rho, \shskip \mu-d}^{(1)} (\widebar u)  + 4^{\mu-\rho} C_{- \mu, \shskip - d}^{\shskip \rho} \cdot G_{\rho, \shskip \mu+d}^{(2)} (u)  G_{\rho, \shskip \mu-d}^{(2)} (\widebar u),
\end{equation}
for all $u$ in the complex plane (the choice of the square root $\sqrt {u}$ is not essential since $\bfF_{\mu,\shskip  d }^{\shskip \rho}$ is even). Let $2 \sqrt u = y e^{i \theta}$. In view of the formulae for $G_{\rho, \shskip \varnu }^{(1)}$ and $G_{\rho, \shskip \varnu }^{(2)}$ in \eqref{2eq: G(1,2) and F(1,2)} and the definition of $\bfF_{\mu,\shskip  d }^{\shskip \rho}$ in \eqref{3eq: defn of F}, it is clear that \eqref{3eq: F = C GG} is equivalent to \eqref{0eq: general W-S, C} in Theorem \ref{thm: general W-S formula}. Finally, thanks to the principal of analytic continuation, it follows from Lemma \ref{lem: convergence of F and G} (1) that the condition $\mu \neq 0$ may be removed and the condition $|\Re \mu | <   \Re \rho < \bfrac 1 2$ may be improved into  $|\Re \mu | <   \Re \rho < 1 $ if $y > 2$. Note that in the definition of $C_{\mu, \shskip d}^{\shskip \rho} $ the zero of $\sin (\pi\mu)$ is annihilated by the pole of $\Gamma  (1 +  \mu + d ) \Gamma  (1 +  \mu - d  )$ at $\mu = 0$ if $d \neq 0$.

\section{Proof of Theorem \ref{thm: regularized W-S, C}}\label{sec: regularized W-S}

Consider % the difference
\begin{equation*}%\label{4eq: defn of F - P}
\bfF_{\mu,\shskip  d }^{\shskip \rho} \big( y e^{i \theta} \big) - \bfP_{\mu,\shskip  d }^{\shskip \rho} \big( y e^{i \theta} \big) =  \int_{0}^{2 \pi} \int_0^\infty  \bfM_{ \mu,\shskip  d }    \big(  x e^{i\phi} \big)   e (- 2 x y \cos (\phi + \theta) )  x^{2 \rho - 1}  d x \shskip d \phi,
\end{equation*}
in which $\bfM_{ \mu,\shskip  d }   (z) = \bfJ_{ \mu,\shskip  d }   (z) - \bfR_{ \mu,\shskip  d }   (z) $ is the regularized Bessel function defined in \eqref{0eq: defn M}.
It is examined in  Lemma \ref{lem: convergence of F and G} that  for the convergence of this integral at infinity we need $\Re \rho < \bfrac 1 2 - |\Re \mu|$ and also  $y >0$. On the other hand, in view of \eqref{2eq: bound for J mu m, weak}, the convergence at zero is secured if $ |\Re \mu | - 1 < \Re \rho $. 

\begin{lem}
The integral $\bfF_{\mu,\shskip  d }^{\shskip \rho} \big( y e^{i \theta} \big) - \bfP_{\mu,\shskip  d }^{\shskip \rho} \big( y e^{i \theta} \big)$ is convergent for $y > 0$ if $|\Re \mu | - 1 <   \Re \rho < \bfrac 1 2 - |\Re \mu |$, uniformly for $\Re \mu$ and $\Re \rho$ in compact sets.
\end{lem}

Let $ C_{\mu, \shskip d}^{\shskip \rho} $ and $F_{\rho, \shskip \varnu }^{(1, \shskip 2)} (z)  $ be defined in \eqref{0eq: defn of C} and \eqref{1eq: defn of F(1,2)}. By Theorem \ref{thm: general W-S formula} and Lemma \ref{lem: formula for G}, we have
\begin{equation}\label{4eq: F-P}
\begin{split}
\bfF_{\mu,\shskip  d }^{\shskip \rho} \big( y e^{i \theta} \big) - \bfP_{\mu,\shskip  d }^{\shskip \rho} \big(  y e^{i \theta} \big)  =  \frac { C_{\mu, \shskip d}^{\shskip \rho}     } {y^{2 \mu + 2\rho} e^{2 i d \theta}} \Big(  F_{\rho, \shskip \mu+d}^{(1)} \big(4 / y^2 e^{2i \theta}\big) \hskip -1  pt F_{\rho, \shskip \mu-d}^{(1)} \big(4 e^{ 2i \theta} / y^2 \big) - 1 \Big) & \\
 +   C_{- \mu, \shskip - d}^{\shskip \rho} \shskip y^{2 \mu - 2\rho} e^{2 i d \theta}   \Big(F_{\rho, \shskip \mu+d}^{(2)} \big(4 / y^2 e^{2i \theta}\big) \hskip -1  pt F_{\rho, \shskip \mu-d}^{(2)} \big(4 e^{2i \theta} / y^2 \big)- 1\Big) &.
\end{split}
\end{equation}
provided that $|\Re \mu | <   \Re \rho < \bfrac 1 2 - |\Re \mu |$. We claim that \eqref{4eq: F-P} remains valid in the extended range $|\Re \mu | - 1 <   \Re \rho < \bfrac 1 2 - |\Re \mu |$. For $d \neq 0$,  the validity of \eqref{4eq: F-P} extends to $|\Re \mu | - 1 <   \Re \rho < \bfrac 1 2 - |\Re \mu |$, since $C_{\mu, \shskip d}^{\shskip \rho}     $ is analytic in this range. Now consider the case $d = 0$. For simplicity, let us assume $\mu \neq 0$. Then $C_{\mu, \shskip 0}^{\shskip \rho} $ has a simple pole at $\rho = - \mu$. However, we have $F_{\rho, \shskip \mu}^{(1)} (z) - 1 \equiv 0 $ if $\rho = -\mu$ and similarly  $F_{\rho, \shskip \mu}^{(2)} (z) - 1 \equiv 0 $ if $\rho = \mu$. So the extension of \eqref{4eq: F-P} to the range $|\Re \mu | - 1 <   \Re \rho < \bfrac 1 2 - |\Re \mu |$ is still permissible. 

Assume that $|\Re \mu | < \bfrac 1 2$. Let $\rho \ra 0$.  Since
\begin{align*}
C_{\mu, \shskip d}^{0} =    {1 }  /  \big({d^2 - \mu^2}\big),
\end{align*}
and 
\begin{align*}%\label{1eq: defn of F(1,2)}
& F_{0, \shskip \varnu }^{(1)} (z^2) = {_2 F_1} \hskip - 1 pt \lp \frac { \varnu }  {2}; \frac  {1 +   \varnu  }  {2}; 1 +\varnu  \shskip; z^2 \rp =   \bigg(  \frac 2 { 1 + \sqrt {1-z^2} }  \bigg)^{ \varnu },  \\
&
F_{0, \shskip \varnu }^{(2)} (z^2) = {_2 F_1} \hskip - 1 pt \lp - \frac { \varnu }  {2};  \frac  {1  - \varnu } {2}; 1 - \varnu  \shskip; z^2  \rp =  \bigg(  \frac   { 1 + \sqrt {1-z^2} } 2 \bigg)^{ \varnu },
\end{align*}
we find that the limiting form of \eqref{4eq: F-P} as $\rho \ra 0$ is exactly \eqref{0eq: special W-S, C} in Theorem \ref{thm: regularized W-S, C}.

\section{Proof of Theorem \ref{thm: Venkatesh} over complex numbers}\label{sec: Proof of Venkatesh}

To start with, we translate the formula \eqref{1eq: main identity} into our language for $F_{\infty} = \BC$. According to the notation of Venkatesh, set $\phi (z) = \varphi (z^2)$ and $\sqrt { \kappa} = \big( k - \sqrt {k^2 - 4} \big) / 2 $ ($|k| > 2$), then the identity we need to prove is the following,
\begin{align}\label{5eq: main identity}
\sideset{}{_{\BC \smallsetminus \{0 \}}}{\iint}  \phi (z) e (\Tr (k z) )  \frac {i d z \nwedge d \widebar z} {|z|^2} = 2 \int_{ -\infty }^{\infty} h (t)  \sinh (\pi t) \left| \frac { k + \sqrt {k^2 - 4} } 2 \right|^{- 2 i t}  \frac { d t} { t} ,
\end{align}
where
\begin{align}\label{5eq: Bessel transform}
\phi (z) = \frac 1 2 \int_{ -\infty }^{\infty}  h (t) \bfJ_{ it} ( z ) \sinh (\pi t) t d t.
\end{align}
Recall here that $h(s)$ is a holomorphic even function in the strip $ |\Im s | \leqslant M$, satisfying decay estimates
\begin{align}\label{4eq: bounds for h (s)}
h (t + i \sigma) \Lt e^{-\pi |t|} (|t|+1)^{- N}.
\end{align}
Define $\omega (z) $ to be
\begin{align}\label{5eq: Mellin transform}
\omega (z) = \frac 1 2 \int_{ -\infty }^{\infty}  h (t) \bfR_{ it} ( z ) \sinh (\pi t) t d t 
= i \int_{ -\infty }^{\infty}  h (t)   {|2 \pi z|^{2 i t } }   \frac { t d t } {\Gamma  (  1 + it )^2 },
\end{align}
where $ \bfR_{ it} (z) = \bfR_{ it, \shskip 0} (z) $ is defined as in (\ref{0eq: defn P mu m (z)}, \ref{0eq: defn of R}). Note that $\omega (z)$ is simply the (horizontal) Mellin inverse transform of the function $ i (2 \pi)^{2it} h (t) t / \Gamma (1+it)^2 $.

We quote from Lemma 4.1 in \cite{Qi-Gauss} the following uniform estimate for $\boldJ_{it} (z)$,
\begin{align}\label{4eq: uniform bound for J}
 t \boldJ_{it} (z) \Lt  ({|t| + 1})^3 \min \big\{ 1, 1/ {  |z|} \big\}  .
\end{align}
Further, it is clear from \cite[3.13 (1)]{Watson} (note that $|\Gamma (1+it)|^2 = \pi t / \sinh (\pi t)$) that
\begin{align}\label{4eq: uniform bound for J-E}
t \lp \boldJ_{it} (z) - \bfR_{ it} (z) \rp \Lt    |z|^2 , \hskip 10 pt |z| \leqslant 1.
\end{align}

We first write the integral on the left of \eqref{5eq: main identity} as follows,
\begin{align}\label{5eq: difference} 
\sideset{}{_{\BC \smallsetminus \{0 \}}}{\iint}  (\phi (z) - \omega (z) ) e (\Tr (k z) )  \frac {i d z \nwedge d \widebar z} {|z|^2} + \sideset{}{_{\BC \smallsetminus \{0 \}}}{\iint}  \omega (z) e (\Tr (k z) )  \frac {i d z \nwedge d \widebar z} {|z|^2}
\end{align}
Applying \eqref{5eq: Bessel transform} and \eqref{5eq: Mellin transform}, the first integral in \eqref{5eq: difference} is equal to
\begin{align*}
 \frac 1 2 \sideset{}{_{\BC \smallsetminus \{0 \}}}{\iint}  \lp \int_{ -\infty }^{\infty}  h (t) \big( \bfJ_{ it} ( z ) - \bfR_{ it} ( z ) \big) \sinh (\pi t) t d t \rp e (\Tr (k z) ) \frac {i d z \nwedge d \widebar z} {|z|^2}. 
\end{align*}
Making crucial use of \eqref{4eq: bounds for h (s)}, \eqref{4eq: uniform bound for J} and \eqref{4eq: uniform bound for J-E}, we verify that the double integral is
absolutely convergent except for the contribution from $ \bfR_{ it} ( z ) $  in the vicinity of $z = \infty$. Nevertheless, we may switch the order of integration. This is because the integral becomes absolutely convergent  after integration by part in a neighborhood of  $z = \infty$ (see Lemma \ref{lem: integration by parts}); we still need \eqref{4eq: bounds for h (s)}  to secure this. We then obtain %in particular, one may switch the order of integration to obtain
\begin{align*}
\frac 1 2    \int_{ -\infty }^{\infty}  h (t) \lp  \sideset{}{_{\BC \smallsetminus \{0 \}}}{\iint} \big( \bfJ_{ it} ( z ) - \bfR_{ it} ( z ) \big) e (\Tr (k z) ) \frac {i d z \nwedge d \widebar z} {|z|^2} \rp \sinh (\pi t) t d t. 
\end{align*}
The inner integral is evaluated in Theorem \ref{thm: regularized W-S, C}. Thus
\begin{equation}\label{4eq: part 1}
\begin{split}
 \sideset{}{_{\BC \smallsetminus \{0 \}}}{\iint}  (\phi (z) - \omega (z) ) & e (\Tr (k z) )  \frac {i d z \nwedge d \widebar z} {|z|^2}  = \\
 & 2 \int_{ -\infty }^{\infty}  h (t) {\textstyle \lp \big|\big( \hskip -1 pt k + \sqrt {k^2 - 4} \big) / 2 \big|^{-2it} - |k|^{-2it} \rp} \sinh (\pi t) \frac {d t} {t}. 
\end{split}
\end{equation}
Now consider the second integral in \eqref{5eq: difference}. Applying  \eqref{5eq: Mellin transform}, we have
\begin{align*}
 \sideset{}{_{\BC \smallsetminus \{0 \}}}{\iint} \hskip -2 pt \omega (z) e (\Tr (k z) )  \frac {i d z \nwedge d \widebar z} {|z|^2} = i \hskip -1 pt \sideset{}{_{\BC \smallsetminus \{0 \}}}{\iint} \hskip -2 pt \lp \int_{ -\infty }^{\infty} \hskip -1 pt  h (t)   {|2 \pi z|^{2 i t } }   \frac { t d t } {\Gamma  (  1 + it )^2 } \hskip -1 pt \rp e (\Tr (k z) )  \frac {i d z \nwedge d \widebar z} {|z|^2}.
\end{align*}
Let $\sigma > 0$ be small, say $\sigma < \bfrac 1 2$. In view of \eqref{4eq: bounds for h (s)}, we may shift the line of integration in the inner
integral to $\Im  t = - \sigma $, then the integral on the right turns into
\begin{align*}  
i \hskip -1 pt \sideset{}{_{\BC \smallsetminus \{0 \}}}{\iint} \hskip -2 pt \lp \int_{- i \sigma -\infty }^{- i \sigma +\infty} \hskip -1 pt  h (t)   {|2 \pi z|^{2 i t } }   \frac { t d t } {\Gamma  (  1 + it )^2 } \hskip -1 pt \rp e (\Tr (k z) )  \frac {i d z \nwedge d \widebar z} {|z|^2}.
\end{align*}
Again, we may switch the order of integration, although this integral is not absolutely convergent in the vicinity of $z = \infty$, obtaining 
\begin{align*}  
i  \hskip -1 pt\int_{- i \sigma -\infty }^{- i \sigma +\infty} \hskip -1 pt  h (t) \lp \sideset{}{_{\BC \smallsetminus \{0 \}}}{\iint} \hskip -2 pt     {|2 \pi z|^{2 i t } }    e (\Tr (k z) )  \frac {i d z \nwedge d \widebar z} {|z|^2} \rp \frac { t d t } {\Gamma  (  1 + it )^2 } .
\end{align*}
The inner integral is evaluated in Lemma \ref{lem: integral of P}  (it is convergent actually as a double integral, since $0 < \Re (it) = \sigma < \bfrac 1 2$). Thus we have
\begin{align*}
\sideset{}{_{\BC \smallsetminus \{0 \}}}{\iint} \hskip -2 pt \omega (z) e (\Tr (k z) )  \frac {i d z \nwedge d \widebar z} {|z|^2} = 2 \int_{- i \sigma -\infty }^{- i \sigma +\infty} \hskip -1 pt  h (t) |k|^{-2it} \sinh (\pi t) \frac {d t} t .
\end{align*}
Utilizing again \eqref{4eq: bounds for h (s)}, we are now free to move the line of integration back to $\Im t = 0 $,
obtaining
\begin{align}\label{4eq: part 2}
\sideset{}{_{\BC \smallsetminus \{0 \}}}{\iint} \hskip -2 pt \omega (z) e (\Tr (k z) )  \frac {i d z \nwedge d \widebar z} {|z|^2} = 2 \int_{ -\infty }^{\infty} \hskip -1 pt  h (t) |k|^{-2it} \sinh (\pi t) \frac {d t} t .
\end{align}
In conclusion, the formula \eqref{5eq: main identity} is proven by summing \eqref{4eq: part 1} and \eqref{4eq: part 2}.

\appendix

\section{}\label{appendix}

In this appendix, we record some extensions of the general Weber-Schafheitlin formula in Theorem \ref{thm: general W-S formula}. These results may be proven by modifying our arguments in \S\S \ref{sec: hypergeometric} and \ref{sec: Proof of Theorem 1}.  

Let $\mu $ be a complex number  and $ d  $ be an integer or half-integer.  Define
\begin{equation*}%\label{0def: J mu m (z)}
	J_{\mu ,\shskip  d} (z) = J_{\mu + d } \lp  z \rp J_{\mu -  d  } \lp  {\widebar z} \rp,
\end{equation*}
and
\begin{equation*}%\label{0eq: defn of Bessel}
\bfJ_{ \mu,\shskip  d} \lp z \rp = 
\left\{ 
\begin{split}
& \frac {1} {\sin (\pi \mu)} \lp J_{-\mu,\shskip  -d} (4 \pi   z) - J_{\mu,\shskip  d} (4 \pi   z)     \rp, \hskip 5 pt \text {if } 2d \text{ is even},\\
& \frac {i} {\cos (\pi \mu)} \lp J_{-\mu,\shskip  -d} (4 \pi   z) + J_{\mu,\shskip  d} (4 \pi   z)  \rp , \hskip 3.3 pt \text {if } 2d \text{ is odd}.
\end{split}
\right.
\end{equation*} We need to take the limit in the non-generic case when  $ \mu - d$ is an integer.
Note that $ \bfJ_{ \mu,\shskip  d} \lp z \rp $ is an even or odd function according as $2d$ is even or odd. 

Let $\rho$ be a complex number and $m $ be an integer. We consider the following integrals
\begin{align*}
& \bfF_{\mu,\shskip  d }^{\shskip \rho, \shskip m} \big( y e^{i \theta} \big) = \int_{0}^{2 \pi} \int_0^\infty \bfJ_{ \mu,\shskip  d}    \big(  x e^{i\phi} \big)  \cos (4 \pi x y \cos (\phi + \theta) )  x^{2 \rho - 1} e^{i m \phi} d x \shskip d \phi, \\
& \bfG_{\mu,\shskip  d }^{\shskip \rho, \shskip m} \big( y e^{i \theta} \big) = \int_{0}^{2 \pi} \int_0^\infty \bfJ_{ \mu,\shskip  d}    \big(  x e^{i\phi} \big)  \sin (4 \pi x y \cos (\phi + \theta) )  x^{2 \rho - 1} e^{i m \phi} d x \shskip d \phi .
\end{align*}
By simple parity considerations, it is clear that the former or the latter integral is zero according as $ m + 2 d$ is odd or even (as the integrand is an odd function). 

We have the following generalization of Theorem \ref{thm: general W-S formula}. 

\begin{thm}
Let notation be as above. 
Define	\begin{align*}
	F^{\lambdaup}_{ \varnu } (z) = {_2 F_1} \hskip - 2 pt \lp   \frac {\lambdaup  +   \varnu } {2}; \frac {1   +   \lambdaup +   \varnu  } {2}; 1   +  \varnu  \shskip; z   \rp .
	\end{align*}
Set 
\begin{equation*}%\label{0eq: defn of C}
C_{ \mu,\shskip  d}^{\shskip \rho, \shskip m} = -  \frac {      \Gamma (\rho +   \mu + m/2  + d) \Gamma ( \rho  +   \mu - m/2 - d)   } {   \Gamma  (1 +  \mu + d ) \Gamma  (1 +  \mu - d  )} \cdot \left\{\begin{split}
& \frac {  \sin (\pi ( \rho  +  \mu))  } {(2 \pi  )^{ 2 \rho} \sin (\pi \mu)}, \hskip 5 pt \text{if } 2d \text{ is even}, \\
&  \frac { \sin (\pi ( \rho  +  \mu))  } {i(2 \pi  )^{ 2 \rho} \cos (\pi \mu)}, \hskip 3.3 pt \text{if } 2d \text{ is odd},
\end{split} \right.
\end{equation*} 
\begin{equation*}%\label{0eq: defn of C}
D_{ \mu,\shskip  d}^{\shskip \rho, \shskip m} =   \frac {      \Gamma (\rho +   \mu + m/2  + d) \Gamma ( \rho  +   \mu - m/2 - d)   } {   \Gamma  (1 +  \mu + d ) \Gamma  (1 +  \mu - d  )} \cdot \left\{\begin{split}
& \frac {  \cos (\pi ( \rho  +  \mu))  } {(2 \pi  )^{ 2 \rho} \sin (\pi \mu)}, \hskip 5 pt \text{if } 2d \text{ is even}, \\
&  \frac {  \cos (\pi ( \rho  +  \mu))  } {i (2 \pi  )^{ 2 \rho} \cos (\pi \mu)}, \hskip 3.3 pt \text{if } 2d \text{ is odd}.
\end{split} \right.
\end{equation*} 
For  all complex $z$, we have
	\begin{align*}
	& \bfF_{\mu,\shskip  d }^{\shskip \rho, \shskip m} (z) = \frac { C_{ \mu,\shskip  d}^{\shskip \rho, \shskip m} F^{\shskip \rho + \frac 1 2 m}_{ \mu+d} \lp 4/z^2 \rp F^{\shskip \rho - \frac 1 2 m}_{ \mu-d} \lp 4/\widebar z^2 \rp } {|z|^{2\rho+2\mu} (z/|z|)^{ m + 2 d} } + \frac { C_{ - \mu,\shskip - d}^{\shskip \rho, \shskip m} F^{\shskip \rho + \frac 1 2 m}_{ - \mu - d} \lp 4/z^2 \rp F^{\shskip \rho - \frac 1 2 m}_{- \mu+d} \lp 4/\widebar z^2 \rp } {|z|^{2\rho-2\mu} (z/|z|)^{m - 2 d} },
	\end{align*}
	if $m + 2 d$ is even, and
	\begin{align*}
	\bfG_{\mu,\shskip  d }^{\shskip \rho, \shskip m} (z) = \frac { D_{ \mu,\shskip  d}^{\shskip \rho, \shskip m} F^{\shskip \rho + \frac 1 2 m}_{ \mu+d} \lp 4/z^2 \rp F^{\shskip \rho - \frac 1 2 m}_{ \mu-d} \lp 4/\widebar z^2 \rp } {|z|^{2\rho+2\mu} (z/|z|)^{ m + 2 d} } + \frac { D_{ - \mu,\shskip - d}^{\shskip \rho, \shskip m} F^{\shskip \rho + \frac 1 2 m}_{ - \mu - d} \lp 4/z^2 \rp F^{\shskip \rho - \frac 1 2 m}_{- \mu+d} \lp 4/\widebar z^2 \rp } {|z|^{2\rho-2\mu} (z/|z|)^{m - 2 d} },
\end{align*}
	if $m + 2 d$ is odd,
of which the former is valid when	$ |\Re \mu| < \Re \rho < \bfrac 1 2 $ and the latter when 	$ |\Re \mu| - \bfrac 1 2 < \Re \rho < \bfrac 1 2 $ {\rm(}we need to take the limit   in the identities when either $\mu = d = 0$ or $\mu =  |d| = \bfrac 1 2${\rm)}. Moreover, for $|z| > 2$, the former identity is valid when	$ |\Re \mu| < \Re \rho < 1 $ and the latter when 	$ |\Re \mu| - \bfrac 1 2 < \Re \rho < 1 $.
\end{thm}

\begin{cor}
	Set 
	\begin{equation*}
	\Delta_{ \mu,\shskip  d}^{\shskip \rho} =  \sin (\pi (\rho + \mu)) \sin (\pi (\rho - \mu)) \cdot \left\{\begin{split}
	& \cos (\pi \rho), \hskip 10 pt \text{if } 2d \text{ is even}, \\
	& i \sin   (\pi \rho), \hskip 9 pt \text{if } 2d \text{ is odd},
	\end{split} \right.
	\end{equation*}
	\begin{equation*}
	\Lambda_{ \mu,\shskip  d}^{\shskip \rho} = i  \cos (\pi (\rho + \mu)) \cos (\pi (\rho - \mu)) \cdot \left\{\begin{split}
	& i \sin (\pi \rho), \hskip 9 pt \text{if } 2d \text{ is even}, \\
	& \cos (\pi \rho), \hskip 10 pt \text{if } 2d \text{ is odd},
	\end{split} \right.
	\end{equation*}
	and
	\begin{align*}
	\Gamma^{\shskip \lambdaup}_{\varnu } = \Gamma \lp  1 / 2 - \lambdaup \rp \Gamma (\lambdaup + \varnu )  \Gamma (\lambdaup - \varnu ) .
	\end{align*}
	We have
	\begin{align*}
	\bfF_{\mu,\shskip  d }^{\shskip \rho, \shskip m} (2) =  \big(2 / \pi^3 \big) (8 \pi)^{-2\rho} \Delta_{ \mu,\shskip  d}^{\shskip \rho}   \Gamma^{\shskip \rho + \frac 1 2 m}_{\mu+d} \Gamma^{\shskip \rho - \frac 1 2 m}_{\mu-d}
	\end{align*}
	for  $ |\Re \mu|  < \Re \rho < \bfrac 1 2 $ and $m+2d$ even, and
	\begin{align*}
	\bfG_{\mu,\shskip  d }^{\shskip \rho, \shskip m} (2) =  \big(2 / \pi^3 \big) (8 \pi)^{-2\rho} \Lambda_{ \mu,\shskip  d}^{\shskip \rho} \Gamma^{\shskip \rho + \frac 1 2 m}_{\mu+d} \Gamma^{\shskip \rho - \frac 1 2 m}_{\mu-d}
	\end{align*}
	for  $ |\Re \mu| - \bfrac 1 2 < \Re \rho < \bfrac 1 2 $ and $m+2d$ odd.

	In particular, for $ |\Re \mu|  < \Re \rho < \bfrac 1 2 $ and half-integer  $d$, we have
	\begin{equation*}%\label{0eq: corollary}
	\begin{split}
	\int_{0}^{2 \pi} \hskip -1 pt \int_0^\infty & \bfJ_{ \mu,\shskip  d}    \big(  x e^{i\phi} \big)  e (- 4 x   \cos \phi )  x^{2 \rho - 1}  d x \shskip d \phi
	= \big(2 / \pi^3 \big) (8 \pi)^{-2\rho} \cos (\pi \rho) \Gamma (1/2 - \rho)^2 \, \cdot \\
	& \hskip -5 pt  \cos (\pi (\rho + \mu)) \cos (\pi (\rho - \mu)) \Gamma (\rho+\mu+d)  \Gamma (\rho+\mu-d) \Gamma (\rho-\mu+d)  \Gamma (\rho-\mu-d),
	\end{split}
	\end{equation*}
	which is an analogue of the formula in Corollary {\rm \ref{cor: d integer}}.
\end{cor}

\end{document}